\documentclass[12pt,reqno]{amsart}

\usepackage[margin=1in]{geometry}
\usepackage{
amsfonts,
amsmath,amssymb,amsbsy,amstext,amsthm,accents,enumerate,float,
verbatim, mathtools, url, 
bm,
mathbbol, 
cases,
shuffle}
\usepackage[dvipsnames]{xcolor}

\usepackage[shortlabels]{enumitem}
\usepackage[colorlinks=true,hyperindex, linkcolor=blue, pagebackref=false, citecolor=cyan]{hyperref}
\usepackage[alphabetic,lite,backrefs]{amsrefs}
\usepackage{lscape}
\usepackage{cleveref}
\Crefname{@theorem}{Theorem}{Theorems}
\crefname{theorem}{Theorem}{Theorems}

\usepackage{graphicx,pstricks}
\usepackage{graphics}
\usepackage{subcaption}
\usepackage{import}
\graphicspath{ {images/} }
\usepackage{subfiles}
\usepackage{float, pdfpages}
\usepackage[normalem]{ulem}

\usepackage[all,cmtip]{xy}
\usepackage{tikz}
\usetikzlibrary{arrows,decorations.markings, cd}
\tikzstyle arrowstyle=[scale=1]
\tikzstyle directed=[postaction={decorate,
decoration={markings,mark=at position .65 with {\arrow[arrowstyle]{stealth}}}}]

\usepackage{mathdots}
\usepackage{bbold} 
\usepackage{tikz-cd}  
\usepackage[shortlabels]{enumitem}
\usepackage{ytableau}
\usepackage[ruled,vlined]{algorithm2e}
\usepackage[all]{nowidow}
\usepackage{genyoungtabtikz}
\newtheorem{lemma}{Lemma}[section]
\newtheorem{theorem}[lemma]{Theorem}

\newtheorem{prop}[lemma]{Proposition}
\newtheorem{cor}[lemma]{Corollary}

\newtheorem{lem-def}[lemma]{Lemma/Definition}
 \newtheorem*{thmA}{Theorem A}
 \newtheorem*{thmA'}{Theorem A'}
 \newtheorem*{thmB}{Theorem B}
 \newtheorem*{thmC}{Theorem C}
\theoremstyle{definition}

\newtheorem{definition}[lemma]{Definition}
\newtheorem{example}[lemma]{Example}

\newtheorem{notation}[lemma]{Notation} 

\newtheorem{construction}[lemma]{Construction} 

\newtheorem{remark}[lemma]{Remark}

\newtheorem*{conventions}{Conventions}

\newcommand{\im}{\operatorname{im}}

\renewcommand{\hom}{\operatorname{Hom}} 

\newcommand{\Sym}{\operatorname{\textbf{Sym}}} 
\newcommand{\GL}{\mathrm{GL}}

\newcommand{\defi}[1]{\textbf{\textsf{#1}}} 

\newcommand{\kk}{\mathbf k}

\newcommand{\cC}{\mathcal{C}}

\newcommand{\cH}{\mathcal{H}}

\newcommand{\NN}{\mathbb{N}}

\newcommand{\QQ}{\mathbb{Q}}

\newcommand{\ZZ}{\mathbb{Z}}

\DeclarePairedDelimiter\abs{\lvert}{\rvert}%

\def\yshort{\ytableaushort}

\newcommand{\set}[1]{\left\{#1\right\}}

\newcommand{\ol}[1]{\overline{#1}}
\newcommand{\bfP}{\mathbf{P}}
\newcommand{\bfC}{\mathbf{C}}

\newcommand{\NSym}{\operatorname{\mathbf{NSym}}} 
\newcommand{\QSym}{\operatorname{\mathbf{QSym}}} 

\newcommand{\ydiagset}{\ytableausetup{mathmode, aligntableaux=center,boxsize=0.5em}}

\newcommand{\multideg}[1]{\mathrm{multideg}}
\newcommand{\Len}{\ell}

\newcommand{\word}{\mathrm{word}}

\newcommand{\diag}{\mathrm{Diag}}
\newcommand{\des}{\mathrm{Des}}
\newcommand{\Comp}{\mathrm{Comp}}
\newcommand{\symm}{\mathfrak{S}}

\newcommand{\weakComp}{\mathrm{weakComp}}

\newcommand{\rowDelete}{\mathrm{rowDelete}}

\newcommand{\sgn}{\mathrm{sgn}}
\newcommand{\cone}{\mathrm{cone}}

\newcommand{\SRT}{\mathrm{SRT}}

\newcommand{\chQ}{\operatorname{ch}^{\QSym}}

\newcommand{\chN}{\operatorname{ch}^{\NSym}}

\newcommand{\rad}{\operatorname{rad}}
\newcommand{\Top}{\operatorname{top}}

\newcommand{\rev}{\operatorname{rev}}

\newcommand{\components}{\operatorname{components}}

\newcommand{\Sh}{\operatorname{Sh}}
\newcommand{\End}{\operatorname{End}}

\begin{document}

\title{Ribbon complexes for the 0-Hecke algebra}

\author{Ayah Almousa}
\address{University of Kentucky}
\email{aalmousa@uky.edu}
\urladdr{\url{https://sites.google.com/view/ayah-almousa}}

\author{Bryan Lu}
\address{University of Washington}
\email{blu17@uw.edu}
\urladdr{\url{https://blu-bird.github.io}}


\date{\today}

\begin{abstract}
We construct explicit tableau-level maps between indecomposable projective modules for the type~$A$
$0$-Hecke algebra that assemble into canonical split short exact sequences lifting the basic ribbon
product rule in $\NSym$ via concatenation and near-concatenation.
Iterating these maps yields cochain complexes indexed by generalized ribbons; we prove these
complexes are acyclic in positive degrees and that their zeroth cohomology is the projective module
indexed by full concatenation.
We apply these complexes, together with VandeBogert's ribbon Schur module criterion, to prove
Koszulness for a naturally defined internally graded algebra object built from the $0$-Hecke tower.
Finally, we define skew projective modules whose noncommutative Frobenius characteristics realize
skewing by fundamental quasisymmetric functions on $\NSym$.
\end{abstract}

\maketitle

\section{Introduction}
The type $A$ $0$-Hecke algebra $\cH_n(0)$ is a well-studied degeneration of the group algebra
$\kk[\mathfrak S_n]$ with a rich combinatorial representation theory: both its simple and indecomposable projective modules are naturally indexed by compositions
(see \cites{Norton-Hecke,Carter-Hecke}).
The Grothendieck groups of the tower $\{\cH_n(0)\}_{n\ge 0}$ form dual
graded Hopf algebras identified with $\QSym$ and $\NSym$ (\cites{KT, bergeron-li}), where the
noncommutative Frobenius characteristic sends the indecomposable projectives
$\bfP_\alpha$ to the ribbon basis $R_\alpha\in \NSym$, while the simples
$\bfC_\alpha$ correspond to the fundamental basis $L_\alpha$ in $\QSym$.
A tableau model for projectives and simples, extending the classical Young tableau approach for
$\mathfrak S_n$, was developed by Huang \cite{heckeTab}.
The goal of this paper is to make this correspondence fully homological by constructing explicit,
canonical, functorial exact sequences and acyclic complexes of projective $0$-Hecke modules that
categorify fundamental ribbon identities in $\NSym$ via concrete tableau-level maps.

\subsection{Acyclic complexes lifting ribbon identities}
Our first goal is to give a representation-theoretic and homological lift of the basic ribbon
product identity
\[
R_\alpha R_\beta \;=\; R_{\alpha\cdot\beta} \;+\; R_{\alpha\odot\beta},
\qquad \alpha\vDash n,\ \beta\vDash m,
\]
together with its iterates.
Rather than working only at the level of Grothendieck groups, we construct explicit,
functorial maps between projective $0$-Hecke modules whose homological behavior reflects
this identity.

\begin{thmA}[\Cref{prop: concatenation-SES,thm: long-ribbon-complex-exact}] For any $\alpha\vDash n$ and $\beta\vDash m$, there is a split short exact sequence
\[
0\to \bfP_{\alpha\cdot\beta}\xlongrightarrow{\ \partial_{\alpha,\beta}\ } \bfP_{(\alpha,\beta)} \xlongrightarrow{\ \mu_{\alpha,\beta}\ } \bfP_{\alpha\odot\beta}\to 0.
\]
More generally, given a sequence $\vec\alpha=(\alpha^{(1)},\dots,\alpha^{(\ell)})$ of compositions, we build an explicit cochain complex $\cC(\vec\alpha)$ whose terms are projectives indexed by generalized ribbons obtained by mixing concatenation and near-concatenation (\Cref{const: longer-concatenation-complex}).
The complex $\cC(\vec\alpha)$ is acyclic in positive degrees and has zeroth cohomology
\[
H^0(\cC(\vec\alpha))\cong \bfP_{\alpha^{(1)}\cdot\alpha^{(2)}\cdots\alpha^{(\ell)}}.
\]
\end{thmA}
Complexes with the same ribbon combinatorics have appeared in several settings.
VandeBogert~\cite{vandebogert2025ribbon} characterizes Koszulness in terms of exactness of canonical
ribbon complexes.
In \cite{almousa2024veronese}, the authors use related Schur-module complexes to build
$\GL_n(\kk)$-equivariant resolutions over Veronese rings, while in~\cite{polishchuk2005quadratic}
similar shapes occur as homogeneous strands of the bar complex in the homological theory of quadratic
(and Koszul) algebras.

Our contribution here is a particularly transparent $0$-Hecke realization of these ideas, in which both the terms and differentials of the complexes admit explicit tableau-level descriptions.
The terms of our complexes are explicit indecomposable projective $0$-Hecke modules indexed
by generalized ribbons, and the differentials are given by simple tableau-level gluing maps.
As a result, the resulting complexes are canonical, functorial in the ribbon data, and admit
a direct combinatorial description that makes acyclicity straightforward to verify.

\subsection{Koszulness}
A second outcome of the existence and functoriality of these ribbon complexes is a Koszulness
statement for a naturally defined \emph{internally graded} algebra built from the tower
$\{\cH_n(0)\}_{n\ge 0}$.
Concretely, in \Cref{sec:koszul} we define an algebra
\[
A \coloneqq \bigoplus_{n\ge 0}\ \bigoplus_{\alpha\vDash n}\bfP_\alpha,
\]
whose multiplication is induced by the near-concatenation maps
$\mu_{\alpha,\beta}:\bfP_{(\alpha,\beta)}\to\bfP_{\alpha\odot\beta}$ and whose internal grading is
given by the number of columns of the ribbon $\diag(\alpha)$.
We then apply a criterion of VandeBogert~\cite{vandebogert2025ribbon} formulated in terms of ribbon
Schur modules to show that this internally graded algebra is Koszul.
\begin{thmB}[\Cref{thm: koszul}]
The internally graded algebra $A$ defined in \Cref{construction: weird-grading-koszul-alg} is Koszul.
\end{thmB}
While Koszulness is a classical homological finiteness property for graded quadratic algebras, its
verification in the present setting is subtle: the algebra $A$ is noncommutative, internally graded,
and defined using induction along the $0$-Hecke tower rather than generators and relations.
Our proof makes essential use of the explicit ribbon-indexed complexes constructed earlier, which help us
identify the relevant ribbon Schur modules and allow VandeBogert’s criterion to be applied
directly.

\subsection{Skew projective modules and branching}
Finally, we introduce skew analogues of the indecomposable projective $0$-Hecke modules.
For $\alpha\vDash n$ and $\beta\vDash k\le n$, we construct a projective $\cH_{n-k}(0)$-module
$\bfP_{\alpha/\beta}$ (see \Cref{def: skew-projective}) whose noncommutative Frobenius characteristic
recovers the skewing operator on $\NSym$:
\[
\chN(\bfP_{\alpha/\beta})=L_\beta^\perp R_\alpha.
\]
Thus skewing identities in $\NSym$ admit a uniform lift from Grothendieck groups to the level of
projective modules.

The case $\beta=(1)$ is especially concrete: $\bfP_{\alpha/(1)}$ models restriction from $\cH_n(0)$ to
the parabolic subalgebra $\cH_{1,n-1}(0)$ and fits into functorial split short exact sequences.
This yields an explicit branching rule in $\NSym$ (\Cref{prop:ses-projective-reps} and
\Cref{cor: nsym-skew-1-box}), so branching is realized by short exact sequences rather than only by
character identities.

\begin{thmC}[\Cref{sec:skew-projectives}]
For compositions $\alpha\vDash n$ and $\beta\vDash k\le n$, there exists a projective
$\cH_{n-k}(0)$-module $\bfP_{\alpha/\beta}$, well defined up to isomorphism, such that
\[
\chN(\bfP_{\alpha/\beta}) \;=\; L_\beta^\perp R_\alpha
\qquad\text{in }\NSym.
\]
In the special case $\beta=(1)$, the modules $\bfP_{\alpha/(1)}$ model restriction from
$\cH_n(0)$ to $\cH_{1,n-1}(0)$ and fit into functorial split short exact sequences of projective
modules.
\end{thmC}

\subsection{Organization}
The paper is organized as follows.
In \Cref{sec:combinatorics} we fix notation for compositions, ribbons, and standard ribbon tableaux.
In \Cref{sec:hopf} we recall the Hopf algebras $\NSym$ and $\QSym$ and the ribbon and fundamental bases.
In \Cref{subsec:zeroHecke} we review the representation theory of $\cH_n(0)$ and the Frobenius characteristic maps relating Grothendieck groups to $\NSym$ and $\QSym$.
In \Cref{sec:concat-complexes} we construct and analyze the acyclic complexes coming from concatenation and near-concatenation, and in \Cref{sec:koszul} we apply these constructions to Koszulness.
Finally, in \Cref{sec:skew-projectives} we define skew projective modules and deduce the branching and skewing formulas described above.

\section{Combinatorics of compositions and ribbons}\label{sec:combinatorics}
In this section, we fix notation for compositions and for (generalized) ribbon diagrams.
We also record the conventions relating compositions, descent sets, and ribbon tableaux that will be used throughout the paper.

\subsection{Compositions and weak compositions}

A \defi{composition} $\alpha = (\alpha_1,\dots, \alpha_k)$ of $n$, denoted
$\alpha\vDash n$, is an ordered list of positive integers whose sum is $n$.
We call the $\alpha_i$ \defi{parts} of $\alpha$.
If we allow some of the parts of $\alpha$ to be zero, we say that $\alpha$ is a
\defi{weak composition}.
Denote by $\Len(\alpha) \coloneqq k$ the \defi{length} of $\alpha$ and
$\abs{\alpha} \coloneqq n$ the \defi{size} of $\alpha$.
We say that a composition $\beta$ \defi{coarsens} $\alpha$ if $\beta$ is obtained
from $\alpha$ by repeatedly summing adjacent parts. Equivalently, $\alpha$
\defi{refines} $\beta$.

Denote by $\Comp$ the set of all compositions, and let $\Comp(n)$ be the set of
all compositions of $n$.
Analogously, denote by $\weakComp$ the set of all weak compositions and by
$\weakComp(n,\ell)$ the set of all weak compositions of $n$ with length $\ell$.

\subsection{Ribbons, generalized ribbons, and descents}

A \defi{ribbon} (= \defi{skew/rim hook} = \defi{border strip}) is a connected
diagram of boxes in $\NN\times\NN$ which does not contain a $2\times 2$ box.

A \defi{generalized ribbon} is a (possibly disconnected) diagram whose connected
components are ribbons, placed with the following convention:
each connected component lies strictly northeast of the previous one by a single
unit along the main diagonal.
Rows and columns of a generalized ribbon are defined globally using the ambient
grid $\NN\times\NN$.

A ribbon is uniquely determined by the number of boxes in each of
its rows. Given a composition $\alpha = (\alpha_1, \dots, \alpha_k)$ of $n$, we
denote by $\diag(\alpha)$ the ribbon with $k$ rows, where the $i$th row counting
from the bottom has $\alpha_i$ boxes.
For $\alpha\in\Comp(n)$, label the cells of $\diag(\alpha)$ with the numbers
$1$ through $n$, starting at the southwestern cell, moving through adjacent
cells, and ending at the northeastern cell. For example, we can label the ribbon
$\diag((3,1,1,2,4))$ as
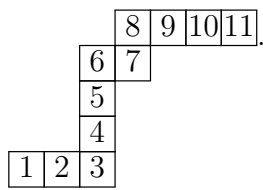
\[
\ytableausetup{boxsize=1.1em}
\begin{ytableau}
\none& \none& \none &8 &9 &{10} &{11}\\
 \none& \none& 6 & 7 \\
 \none& \none& 5  \\
 \none& \none&4    \\
   1&2 &3    \\
\end{ytableau}.
\]

If the cell labeled $i+1$ lies strictly above the cell labeled $i$, we say that
the ribbon has a \defi{descent} at $i$. The resulting subset of $[n-1]$ is called
the \defi{descent set} of the ribbon.
In the example above, the descent set is
$\des(3,1,1,2,4)=\{3,4,5,7\}$.

For each subset $R \subseteq [n-1]$, there is exactly one ribbon with $n$ cells
and descent set $R$.
If $R\subseteq [n-1]$, denote by $\diag(R)$ the unique ribbon with descent
set $R$. We will often abuse notation and write only the composition corresponding
to $\diag(R)$. For example, we might write $\diag(\{3,4,5,7\}) = (3,1,1,2,4).$

\subsection{Standard ribbon tableaux and reading words}

A \defi{tableau} is a bijective filling of a diagram with the numbers
$\{1,\dots, n\}$.
A tableau on a (generalized) ribbon shape is \defi{standard} if its entries
strictly increase along rows from left to right and strictly increase along
columns from top to bottom, wherever two cells are horizontally or vertically
adjacent. Note that disconnected components impose no additional conditions.

Denote by $\SRT(\alpha)$ the set of standard ribbon tableaux of shape
$\diag(\alpha)$.
Every element of $\SRT(\alpha)$ is naturally in bijection with a permutation in
$\symm_n$ whose descent set coincides with the shape descent set $\des(\alpha)$.

Given a standard (generalized) ribbon tableau $T$, define its \defi{reading word}
$\word(T)$ as follows: read the entries of each connected component from
southwest to northeast along adjacent cells, and concatenate the resulting words
by ordering components from southwest to northeast.

For a permutation $w\in \symm_n$ written in one-line notation, write $\SRT(w)$ for
the set of standard ribbon tableaux of shape $\diag(\des(w))$.
For example, if $w = 451362\in \symm_6$, the corresponding ribbon tableau is
\[
\ytableausetup{centertableaux, boxsize=1em}
\yshort{\none\none\none 2, \none 136,45}
\]
and the associated ribbon shape is $\diag(w)=(2,3,1)$.
We say that $w$ is the reading word of the tableau $T$, and we write
$\word(T)=w$. With this convention, the map $T\mapsto \word(T)$ is a bijection between $\SRT(\alpha)$ and the set of
permutations $w\in\symm_n$ with $\des(w)=\des(\alpha)$: along the ribbon traversal, horizontal adjacencies
force rises and vertical adjacencies force descents.

Given a descent set $I\subseteq [n-1]$ or a composition $\alpha$, define
$w_0(I)=w_0(\alpha)$ to be the word corresponding to the tableau
$\tau_0(I)=\tau_0(\alpha)$ which fills each \emph{column} of the ribbon
corresponding to $\alpha$ with the numbers $1$ through $n$, read from top to
bottom and left to right.
Define $w_1(I)=w_1(\alpha)$ analogously using the tableau $\tau_1(I)=\tau_1(\alpha)$
which fills each \emph{row} from left to right, starting from the top row and
moving downward.

For a standard (generalized) ribbon tableau $T$, define its \defi{descent set} by
\[
\des(T)\coloneqq \{\, i\in[n-1] : \text{$i+1$ lies strictly above $i$ in $T$}\,\}.
\]

\subsection{Operations on compositions}

Given a composition $\alpha = (\alpha_1, \dots, \alpha_\ell)$, we can denote the \defi{reverse} of $\alpha$ as $\rev(\alpha) \coloneqq (\alpha_\ell,\dots, \alpha_1)$. The \defi{complement} of $\alpha$, denoted by $\alpha^c$, is the unique composition of $n$ whose descent set is $[n-1] \setminus \des(\alpha)$. The \defi{transpose} or \defi{conjugate} of $\alpha$ is $\alpha^\top \coloneqq \rev(\alpha^c) = (\rev(\alpha))^c$, or equivalently, the transpose of the ribbon diagram of $\alpha$. Note that we always have that $w_1(\alpha)^{-1} = w_1(\rev(\alpha))$.

\begin{example}\label{ex: tau-examples}
Let $\alpha = (2,3,1)$. Then we have:
\begin{gather*}
\tau_0(\alpha) = \yshort{\none\none\none 5, \none 246,13}\ , \; \tau_1(\alpha) = \yshort{\none\none\none 1,\none 234,56}\ , 
\tau_0(\rev(\alpha)) = \yshort{\none\none 46,135,2}, \, 
\tau_1(\rev(\alpha)) = \yshort{\none\none 12,345,6}
\\
\tau_0(\alpha^c) = \yshort{\none 56, \none 4, 23, 1}\ , \; \tau_1(\alpha^c) = \yshort{\none 12,\none 3, 45, 6}, \;  
\tau_0(\alpha^\top) = \yshort{\none\none 5, \none 26, \none 3, 14}\ , \; \tau_1(\alpha^\top) = \yshort{\none\none 1, \none 23, \none 4, 56}
\end{gather*}
In particular, note that $w_1(\alpha) = 562341$ and $w_1(\rev(\alpha)) = 634512 = w_1(\alpha)^{-1}$.
\end{example}

\begin{definition}\label{def: concatenation-near-concat} 
For two compositions $\alpha =(\alpha_1,\dots, \alpha_\ell)$ and $\beta = (\beta_1,\dots, \beta_m)$, define their \defi{concatenation} $\alpha\cdot \beta$ and their \defi{near-concatenation} $\alpha\odot\beta$ by:
\begin{align*}
    \alpha\cdot \beta &\coloneqq (\alpha_1,\dots, \alpha_\ell, \beta_1,\dots, \beta_m),\\
    \alpha\odot\beta &\coloneqq (\alpha_1,\dots, \alpha_{\ell-1}, \alpha_{\ell}+\beta_1, \beta_2,\dots, \beta_m)
\end{align*}
\end{definition}

\begin{conventions}\label{conventions: diagrams}
Throughout this paper, we will use the following conventions. Let $\diag(\alpha\cdot\beta)$ be the connected ribbon of size $n+m$.
Inside $\diag(\alpha\cdot\beta)$ there is a canonical decomposition into two subribbons:
the \emph{lower} subribbon consisting of the first $n$ cells along the southwest-to-northeast ribbon traversal (isomorphic to $\diag(\alpha)$),
and the \emph{upper} subribbon consisting of the last $m$ cells (isomorphic to $\diag(\beta)$).

We fix the following placement convention for the disconnected diagram $\diag(\alpha,\beta)$:
it is obtained from $\diag(\alpha\cdot\beta)$ by translating the upper $\beta$-subribbon \emph{one unit to the right}.
Thus $\diag(\alpha,\beta)$ has two connected components, and its $\alpha$-component and $\beta$-component are separated by a unit gap.
Similarly, $\diag(\alpha\odot\beta)$ is obtained from $\diag(\alpha,\beta)$ by translating the $\beta$-component \emph{one unit down}.
With this convention, in $\diag(\alpha\odot\beta)$ the northeasternmost cell of the $\alpha$-component becomes horizontally adjacent to (and immediately left of)
the southwesternmost cell of the $\beta$-component.
\end{conventions}

\begin{example} 
Let $\alpha = (3,2)$ and $\beta = (3,1,1)$. Then
\begin{align*}
\ytableausetup{centertableaux, boxsize=0.5em}
    \diag(\alpha\cdot\beta) &= \ydiagram{7+1,7+1,5+3,4+2,2+3}*[*(gray)]{7+1,7+1,5+3}, &\quad
    \diag(\alpha\odot\beta) &= 
\ydiagram{8+1,8+1,4+5,2+3}*[*(gray)]{8+1,8+1,6+3},
    &\quad
    \diag(\alpha, \beta) &= 
    \ydiagram{8+1,8+1,6+3,4+2,2+3}*[*(gray)]{8+1,8+1,6+3}.
\end{align*}    
\end{example}

\begin{definition}[Shuffle product]
Let $A=\kk\langle x_1,\dots,x_n\rangle$ be the free associative algebra on
$n$ letters, with basis given by words in the alphabet
$\{x_1,\dots,x_n\}$.
For two words
\[
u = u_1 u_2 \cdots u_p
\quad\text{and}\quad
v = v_1 v_2 \cdots v_q,
\]
the \emph{shuffle product} of $u$ and $v$ is defined by
\[
u \shuffle v
\;\coloneqq\;
\sum_{\sigma \in \Sh(p,q)}
w_{\sigma^{-1}(1)}\, w_{\sigma^{-1}(2)} \cdots w_{\sigma^{-1}(p+q)},
\]
where
\[
(w_1,\dots,w_{p+q})=(u_1,\dots,u_p,\,v_1,\dots,v_q),
\]
and $\Sh(p,q)$ denotes the set of $(p,q)$-shuffles, i.e.\ permutations
$\sigma\in S_{p+q}$ such that
\[
\sigma(1)<\cdots<\sigma(p)
\quad\text{and}\quad
\sigma(p+1)<\cdots<\sigma(p+q).
\]
The shuffle product is extended $\kk$-bilinearly to all of $A$.
\end{definition}

\begin{example}
Let $A=\kk\langle x_1,\dots,x_n\rangle$.  We wish to compute $(x_1x_2)\shuffle x_3$.
Here $p=2$ and $q=1$, so $\Sh(2,1)$ consists of the three shuffles placing the
single letter $x_3$ into one of the three slots among $x_1,x_2$ while
preserving the relative order of $x_1,x_2$. Hence
\[
(x_1x_2)\shuffle x_3
=
x_1x_2x_3 \;+\; x_1x_3x_2 \;+\; x_3x_1x_2.
\]
\end{example}

\section{Symmetric, noncommutative symmetric, and quasisymmetric functions}\label{sec:hopf}
In this section, we recall relevant facts about the Hopf algebras $\Sym$, $\NSym$, and $\QSym$. A helpful reference for this section is \cite{grinbergReiner-hopf-survey}.

\subsection{Symmetric functions}

The \defi{algebra of symmetric functions} $\Sym$ is the commutative polynomial algebra $\QQ[h_1,h_2,\dots]$ where $\deg(h_n) = n$. The $h_n$ are the \defi{complete homogeneous} generators. Given an integer $n\geq 0$ and a partition $\lambda = (\lambda_1 \geq \lambda_2\geq \dots \geq \lambda_k)$ of $n$, the \defi{complete homogeneous symmetric function} is the element of $\Sym$ given by
\[
h_\lambda \coloneqq h_{\lambda_1}\dots h_{\lambda_k},
\]
where we use the convention $h_\emptyset = 1$. Other important bases of $\Sym$ include the elementary symmetric functions $e_\lambda$, the monomial symmetric functions $m_\lambda$, and the Schur functions $s_\lambda$.

The algebra $\Sym$ is a self-dual Hopf algebra, which has the following pairing called the \defi{Hall inner product}:
\[
\langle h_\lambda, m_\mu\rangle = \delta_{\lambda,\mu}
\]
for all partitions $\lambda$ and $\mu$.
An element $f\in\Sym$ gives rise to an operator $f^\perp: \Sym\to\Sym$ according to the relation
\[
\langle fg,h\rangle = \langle g,f^\perp h\rangle \quad \text{for all } g,h\in\Sym.
\]

\subsection{Noncommutative symmetric functions}

The \defi{algebra of noncommutative symmetric functions} $\NSym$ is a noncommutative analogue of $\Sym$: it is the free associative algebra on generators $\{H_1,H_2,\dots\}$. Given a composition $\alpha = (\alpha_1,\dots, \alpha_k)$, the \defi{complete homogeneous function} associated to $\alpha$ is defined as
\[
H_\alpha \coloneqq H_{\alpha_1} H_{\alpha_2}\dots H_{\alpha_k}.
\]

Another basis for $\NSym$ is given by the \defi{noncommutative ribbon functions} $\{R_\alpha\}_{\alpha\in\Comp}$. These two bases of $\NSym$ satisfy the following relations:
\begin{align}
H_\alpha &= \sum_{\beta\text{ coarsens }\alpha} R_\beta, \label{eq: H-to-R}\\
R_\alpha &= \sum_{\beta \text{ coarsens } \alpha} (-1)^{\Len(\beta)-\Len(\alpha)} H_\beta. \label{eq: R-to-H}
\end{align}

The ribbon basis also satisfies the following product identity:
\begin{equation}\label{eq: product-of-two-ribbons}
R_\alpha R_\beta  = R_{\alpha \cdot \beta} + R_{\alpha \odot \beta}.
\end{equation}

The \defi{forgetful map} $\chi: \NSym\to \Sym$ is defined by sending the element $H_\alpha$ to the complete homogeneous symmetric function
\[
\chi(H_\alpha) \coloneqq h_{\alpha_1}h_{\alpha_2}\cdots h_{\alpha_k}\in\Sym
\]
and extending linearly. Note that, by the Jacobi--Trudi identity, this map sends $R_\alpha$ to the skew Schur function $r_\alpha \coloneqq s_{\diag(\alpha)}$ corresponding to the ribbon diagram $\diag(\alpha)$.

\subsection{Quasisymmetric functions}

The \defi{algebra of quasisymmetric functions} $\QSym$ is the graded Hopf algebra dual to $\NSym$ which contains $\Sym$ as a subalgebra. Let $\{M_\alpha\}_{\alpha\in\Comp}$ be the monomial quasisymmetric basis, which is dual to the $\{H_\alpha\}$ basis of $\NSym$. The \defi{fundamental quasisymmetric function} $L_\alpha$ is dual to the ribbon function $R_\alpha$ of $\NSym$. These two bases are related by
\[
L_\alpha = \sum_{\beta\text{ refines } \alpha} M_\beta.
\]

The structure maps for the Hopf algebra $\QSym$ in the basis $\{L_\alpha\}_{\alpha\in\Comp}$ of fundamental quasisymmetric functions are as follows:
\begin{itemize}
    \item Coproduct:
    \[
    \Delta L_\alpha = \sum_{\substack{(\beta,\gamma):\\ \beta\cdot\gamma = \alpha \text{ or } \beta\odot\gamma = \alpha}} L_\beta \otimes L_\gamma,
    \]
\item Product:
Let $n=\abs{\alpha}$ and $m=\abs{\beta}$. Choose any permutations $u\in\symm_n$ and $v\in\symm_m$
with descent compositions $\diag(u)=\alpha$ and $\diag(v)=\beta$.
Let $v^{+n}$ be the word obtained from $v$ by adding $n$ to each letter, so $v^{+n}$ is a permutation of
$\{n+1,\dots,n+m\}$.
Then
\[
L_\alpha L_\beta \;=\; \sum_{w\in u \shuffle v^{+n}} L_{\diag(w)}.
\]

\end{itemize}
Here we use the following notation:
\begin{itemize}
    \item $w_\alpha$ is any reading word for any standard tableau in $\SRT(\alpha)$;
    \item $w_\beta$ is any reading word of a standard tableau in $\SRT(\beta)$;
    \item $u\shuffle v^{+n}$ denotes the set of shuffles of the words $u$ and $v^{+n}$;
    \item $\diag(w)$ is the descent composition of the permutation $w$, or equivalently, the composition corresponding
    to $\des(w)$.
\end{itemize}

\begin{example}
To multiply $L_{(2)}L_{(1,1)}$, let $w_{(2)} = 12$ corresponding to the diagram $
\ytableausetup{boxsize=1em}\yshort{12}$ 
and $w_{(1,1)} = 43$ corresponding to the diagram $\yshort{3,4}$. Then
\begin{align*}
\ytableausetup{boxsize = 0.3em}
    L_{(2)}L_{(1,1)} &= \sum_{w\in 12\shuffle 43} L_{\diag(w)} \\ &= L_{\diag(1243)} &+& L_{\diag(1423)} &+& L_{\diag(1432)} &+& L_{\diag(4123)} &+& L_{\diag(4132)} &+&
    L_{\diag(4312)} \\
    &= L_{\begin{ytableau}
        \none&\none&\\ & &
    \end{ytableau}}
    &+&
    L_{\begin{ytableau}
    \none&&\\&&\none&\none
    \end{ytableau}}
    &+&
    L_{\begin{ytableau}
        \none&\\ \none&\\ &
    \end{ytableau}}
    &+& 
L_{\ydiagram{3,1}}
    &+&
L_{\begin{ytableau}
        \none&\\
        &\\
        &\none
    \end{ytableau}}
    &+& 
L_{\ydiagram{2,1,1}}.
\end{align*}
\end{example}

\subsection{Pairing between \texorpdfstring{$\NSym$}{NSym} and \texorpdfstring{$\QSym$}{QSym}}

The algebras $\NSym$ and $\QSym$ have a pairing $\langle\cdot,\cdot\rangle : \NSym \times \QSym \to \QQ$, defined under this duality by
\[
\langle R_\alpha, L_\beta\rangle = \delta_{\alpha, \beta}.
\]
The operation adjoint to multiplication with respect to the scalar product on $\Sym$ can be generalized to $\NSym$ and $\QSym$ by using this pairing: for $F\in \QSym$, denote $F^\perp$ to be the operator which acts on elements $H\in \NSym$ according to the relation
\begin{equation}\label{eq: nsym-qsym-pairing}
    \langle H,FG\rangle = \langle F^\perp H, G\rangle \quad \text{for all } G\in \QSym.
\end{equation}
Therefore, given $H\in \NSym$, we can compute $F^\perp(H)$ via the formula
\begin{equation}\label{eq:skew-operators-general}
F^\perp(H) = \sum_{\alpha\in\Comp} \langle H,F L_\alpha\rangle R_\alpha.
\end{equation}

\section{Representation theory of the type A 0-Hecke algebra}\label{subsec:zeroHecke}
In this section, we recall the basics of the representation theory of the $0$-Hecke algebra and how that representation theory connects to the Hopf algebras $\NSym$ and $\QSym$.
Recall the symmetric group $\symm_n$ on $n$ elements is generated by $\set{s_1, \dots, s_{n-1}}$, where $s_i$ denotes the transposition $(i, i+1)$ for $1 \leq i < n$. These transpositions satisfy relations $s_i^2 = 1$, $(s_i s_{i+1})^3 = 1$ for $1 \leq i < n-1$, and $(s_i s_j)^2 = 1$ otherwise, which are equivalent to the following relations:
\begin{align*}
     s_i^2 &= 1 & 1 \leq i < n, \\ 
     s_i s_{i+1} s_i &= s_{i+1} s_i s_{i+1} & 1 \leq i < n -1, \\
     s_i s_j &= s_j s_i & |i - j| > 1,\; 1 \leq i, j < n.
\end{align*}
The 0-Hecke algebra (of type A) on $n$ elements is a deformation of the group algebra of $\symm_n$, with generators satisfying relations that resemble those of the underlying Coxeter group.

\begin{definition}
    Consider generators $\pi_1, \dots, \pi_{n-1}$ satisfying the following relations: 
    \begin{align*}
        \pi_i^2 &=  \pi_i & 1 \leq i < n, \\
        \pi_i \pi_{i+1}  \pi_i &= \pi_{i+1}  \pi_i \pi_{i+1} & 1 \leq i < n-1, \\
        \pi_i \pi_j &= \pi_j \pi_i & |i-j| > 1,\; 1\leq i, j < n.
    \end{align*}
    Let $\kk$ be any field. The \defi{0-Hecke algebra} on $n$ elements, $\cH_n(0)$, is the $\kk$-algebra generated by $\set{\pi_1, \dots, \pi_{n-1}}$. Set $\ol\pi_i = \pi_i-1$; then $\ol\pi_i\pi_i = 0$, and one can check that $\cH_n(0)$ is generated by $\{\ol\pi_1,\dots, \ol\pi_{n-1}\}$ with the same relations as for the $\pi_i$ above, except that $\ol\pi_i^2 = -\ol \pi_i$.
\end{definition}

The representation theory of $\cH_n(0)$ has been studied extensively in \cite{Norton-Hecke}, \cite{Carter-Hecke}, \cite{KT}, and more recently \cite{heckeTab}. In particular, the representation theory of $\cH_n(0)$ is intimately related to $\NSym$ and $\QSym$. We outline this relationship in the next subsection.

By \cite{Norton-Hecke}, compositions $\alpha \vDash n$ 
index both the \textit{simple} $\cH_n(0)$-modules, denoted $\bfC_\alpha$, and the \textit{indecomposable projective representations} of $\cH_n(0)$, denoted $\bfP_\alpha$. Recall that the \defi{radical} $\rad(M)$ of a module $M$ is the intersection of all maximal submodules of $M$, and the \defi{top} of $M$ is $\Top(M) \coloneqq M/\rad(M)$, which is always semisimple. In our situation, we have $\Top(\bfP_\alpha) = \bfC_\alpha$.

An explicit construction for the simple and indecomposable projective $\cH_n(0)$-modules can be found in \cite{Norton-Hecke}. We consider another construction in terms of tableaux described in \cite{heckeTab}.

\begin{definition}[\cite{heckeTab}]\label{def: hecke-action}
    For a composition $\alpha \vDash n$, the vector space $\bfP_\alpha$ has basis consisting of standard ribbon tableaux of shape $\diag(\alpha)$. For $i \in [n-1]$ and $T$ a standard ribbon tableau of shape $\diag(\alpha)$, we define an action of $\cH_n(0)$ on $\bfP_\alpha$ by
    \[
    \ol \pi_i (T) := \begin{cases}
        -T & \text{if $i$ is in a higher row of $T$ than $i+1$}, \\ 
        0 & \text{if $i$ is in the same row of $T$ as $i+1$}, \\ 
        s_i(T) & \text{if $i$ is in a lower row of $T$ than $i+1$},
    \end{cases}
    \]
    where $s_i(T)$ denotes the tableau obtained from $T$ by swapping the entries $i$ and $i+1$.
\end{definition}

A \emph{generalized ribbon} corresponds to a sequence of compositions $\vec \beta = (\beta^{(1)}, \beta^{(2)}, \dots, \beta^{(\ell)})$, where each $\beta_i \vDash n_i$, and $\vec \beta$ corresponds to a skew diagram $\diag(\vec \beta)$ with disconnected ribbon components $\diag(\beta^{(i)})$ from bottom to top. Such generalized ribbons correspond to projective representations of $\cH_n(0)$, for $n = \sum_{i=1}^\ell n_i$, given by 
\[
\bfP_{\vec \beta} \cong \left[\bigotimes_{i=1}^\ell \bfP_{\beta^{(i)}} \right] {\Bigg\uparrow}_{\cH_{n_1, \dots, n_\ell}(0)}^{\cH_n(0)},
\]
$\cH_{n_1,\dots,n_\ell}(0) \cong \cH_{n_1}(0)\otimes_\kk \cdots \otimes_\kk \cH_{n_\ell}(0)$. The projective module $\bfP_{\vec\beta}$ has basis elements consisting of standard tableaux of shape $\diag(\vec \beta)$ and the same action as above. 

\begin{example}
    If $\alpha = (1, 2, 1)$, consider the standard tableaux of shape $\diag(\alpha) = \ydiagset \ydiagram{1+1,2,1}$:
    \ytableausetup{centertableaux, boxsize = 1em}
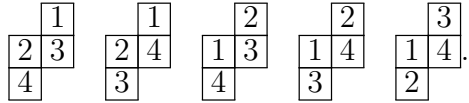
    \[
        \ytableaushort{\none1,23,4} \quad \ytableaushort{\none1,24,3} \quad 
        \ytableaushort{\none2,13,4} \quad \ytableaushort{\none2,14,3} \quad \ytableaushort{\none3,14,2}.
    \]
    We can draw the diagram showing the action of $\cH_4(0)$ on $\bfP_{(1,2,1)}$: 
    
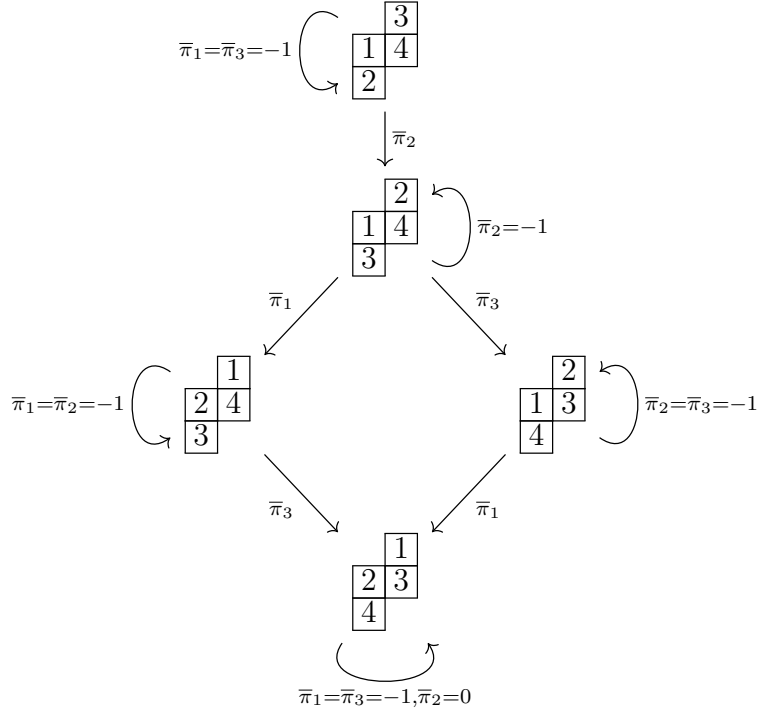
\begin{center}
        \begin{tikzcd}
 & {\ytableaushort{\none3,14,2}} \arrow[d, "\ol\pi_2"] \arrow["\ol\pi_1 = \ol\pi_3 = -1"', loop, distance=2em, in=215, out=145]   & \\ 
 & {\ytableaushort{\none2,14,3}} \arrow[ld, "\ol\pi_1"'] \arrow[rd, "\ol\pi_3"] \arrow["\ol\pi_2 = -1"', loop, distance=2em, in=35, out=325] & \\
{\ytableaushort{\none1,24,3} } \arrow[rd, "\ol\pi_3"'] \arrow["\ol\pi_1 = \ol\pi_2 = -1"', loop, distance=2em, in=215, out=145] &  & {\ytableaushort{\none2,13,4}} \arrow[ld, "\ol\pi_1"] \arrow["\ol\pi_2 = \ol\pi_3 = -1"', loop, distance=2em, in=35, out=325] \\
 & {\ytableaushort{\none1,23,4}} \arrow["{\ol\pi_1 = \ol\pi_3 = -1, \ol\pi_2 = 0}"', loop, distance=2em, in=305, out=235]  & 
\end{tikzcd}
    \end{center}

Here, $\ol\pi_i = -1$ on an arrow means that $\ol\pi_i$ acts on that basis element by multiplication by $-1$, and similarly $\ol\pi_i = 0$ on an arrow means that $\ol\pi_i$ takes that basis element to $0$. 
\end{example}

\begin{theorem}[\cite{heckeTab}]
    For $\alpha \vDash n$, the corresponding $\cH_n(0)$-module $\bfP_\alpha$ with the action described above is isomorphic to the projective indecomposable representation of $\cH_n(0)$ corresponding to $\alpha$ defined by Norton \cite{Norton-Hecke}.
\end{theorem}

From \cite{KT}, we see that the noncommutative Frobenius characteristic $\mathbf{ch}$ sends $\bfP_\alpha$ to $R_\alpha$, the noncommutative ribbon functions. These $\bfP_\alpha$ are then natural analogues of the Specht modules corresponding to ribbon diagrams in the representation theory of $\symm_n$, as their characters $R_\alpha$ in $\NSym$ map surjectively onto the symmetric ribbon Schur functions $r_\alpha$ in $\Sym$.

We can also use the tableau approach to define the simple modules of $\cH_n(0)$.

\begin{definition}[\cite{heckeTab}*{\S~3.2}]
Let $\alpha \in \Comp(n)$, and let $\tau_1(\alpha)\in\SRT(\alpha)$ be the standard ribbon tableau that
fills each row from left to right, starting from the top row and moving downward (as in
\Cref{sec:combinatorics}). Define the $1$-dimensional $\cH_n(0)$-module
\[
\bfC_\alpha \coloneqq \kk \cdot \tau_1(\alpha),
\]
where $\ol\pi_i$ acts by the scalar prescribed by \Cref{def: hecke-action}. Then
$\bfC_\alpha \cong \Top(\bfP_\alpha)$, and the reading word of the generator is $w_1(\alpha)$.
\end{definition}

\begin{remark}
For the $0$-Hecke algebra, a $1$-dimensional module $M=\kk v$ is determined by its
\emph{descent set}
\[
D(M)\coloneqq \{\, i\in[n-1] : \ol\pi_i v = -v \,\}.
\]
By Norton~\cite{Norton-Hecke}, the simple module $\bfC_\alpha$ is the unique
$1$-dimensional $\cH_n(0)$-module with $D(\bfC_\alpha)=\des(\alpha)$.
In the tableau model of \Cref{def: hecke-action}, the span $\kk\cdot T$ is a $\cH_n(0)$-submodule
if and only if $\ol\pi_i(T)$ is always a scalar multiple of $T$, i.e., if no generator acts by
a nontrivial swap $T\mapsto s_i(T)$.
The tableau $\tau_1(\alpha)$ has this property, so $\kk\cdot \tau_1(\alpha)$ realizes
$\bfC_\alpha$.
\end{remark}

\subsection{Relating Grothendieck groups of the 0-Hecke algebra to \texorpdfstring{$\NSym$}{NSym} and \texorpdfstring{$\QSym$}{QSym}}
For each $n\ge 0$, let $G_0(\cH_n(0))$ be the Grothendieck group of the category of
finite-dimensional $\cH_n(0)$-modules, i.e., the free abelian group on isomorphism classes
$[M]$ modulo the relations $[M]=[L]+[N]$ for every short exact sequence
$0\to L\to M\to N\to 0$.
Let $K_0(\cH_n(0))$ be the analogous Grothendieck group for the exact category of
finite-dimensional \emph{projective} $\cH_n(0)$-modules.
We then set
\[
G_0(\cH_\bullet(0))\coloneqq \bigoplus_{n\ge 0}G_0(\cH_n(0)),
\qquad
K_0(\cH_\bullet(0))\coloneqq \bigoplus_{n\ge 0}K_0(\cH_n(0)).
\]
Krob and Thibon \cite{KT} showed that $G_0(\cH_\bullet(0))$ and $K_0(\cH_\bullet(0))$ are isomorphic to $\QSym$ and $\NSym$, respectively, as graded algebras. They did this by introducing the following maps.

\begin{definition}
Let $M$ be a finite-dimensional $\cH_n(0)$-module and consider a composition series for $M$, that is, a decreasing sequence
\[
M = M_1 \supset M_2 \supset \dots \supset M_k \supset M_{k+1} = 0
\]
where the successive quotients $M_i/M_{i+1}$ are simple. Then each $M_i/M_{i+1}$ is isomorphic to some $\bfC_{\alpha_i}$, and the Jordan–Hölder theorem ensures that the quasisymmetric function
\begin{equation}
\chQ(M) = \sum_{i=1}^k L_{\alpha_i}
\end{equation}
is independent of the choice of composition series. We call $\chQ$ the \defi{quasisymmetric Frobenius characteristic} of $M$.

Now, suppose that $M$ is a finite-dimensional projective $\cH_n(0)$-module. Then $M$ is isomorphic to a direct sum of indecomposable projective modules
\[
M = \bigoplus_{i=1}^m \bfP_{\alpha_i}.
\]
The \defi{noncommutative Frobenius characteristic} of $M$ is the noncommutative symmetric function $\chN$ defined by
\begin{equation}
\chN(M) \coloneqq \sum_{i=1}^m R_{\alpha_i}. 
\end{equation}
\end{definition}

In analogy with the theory of representations of $\symm_n$, Krob and Thibon \cite{KT} show that $\chQ: G_0(\cH_\bullet(0)) \to \QSym$ and $\chN: K_0(\cH_\bullet(0)) \to \NSym$ give isomorphisms of graded algebras. Bergeron and Li \cite{bergeron-li} show that in fact $K_0(\cH_\bullet(0))$ and $G_0(\cH_\bullet(0))$ have the structure of Hopf algebras, and that the maps $\chQ$ and $\chN$ respect this structure. For more details of this proof, the interested reader should consult \cite{bergeron-li}*{Example 4.3}.

For any module $M$ over a field $\kk$, we have that
$\dim_\kk \hom_{\cH_n(0)}\left(\bfP_\alpha, M\right)$ is the multiplicity of $\bfC_\alpha$ as a composition factor of $M$ \cite{Benson1991}*{Lemma 1.7.6}. Therefore, for any $\kk$, there is a natural pairing
\[
K_0(\cH_\bullet(0)) \times G_0(\cH_\bullet(0)) \to \kk
\]
given by
\begin{equation}\label{eq: pairing-on-0-hecke-reps}
\langle \bfP_\alpha, \bfC_\beta\rangle = \dim_\kk \hom_{\cH_n(0)}\left(\bfP_\alpha, \bfC_\beta\right) = \delta_{\alpha,\beta}.
\end{equation}
This pairing corresponds to the pairing between $\NSym$ and $\QSym$ described in \Cref{eq: nsym-qsym-pairing} in the way one would expect.

\section{Acyclic complexes from concatenation and near-concatenation}\label{sec:concat-complexes}

In this section, we introduce acyclic complexes of projective $0$-Hecke modules that categorify the ribbon identity
$R_\alpha R_\beta = R_{\alpha\cdot \beta} + R_{\alpha \odot \beta}$ in $\NSym$.
We will use the existence of these complexes and a result of VandeBogert~\cite{vandebogert2025ribbon}
to conclude in the next section that the internally graded algebra $A$ from
\Cref{construction: weird-grading-koszul-alg} (whose multiplication is induced by near-concatenation)
is Koszul.
Very similar complexes also appeared in \cite{polishchuk2005quadratic}*{\S 8.2} as homogeneous strands of the bar complex, as well as \cite{almousa2024veronese} for ribbon Schur functors, although the maps here are much simpler.

Throughout this section, if $\alpha\vDash n$ and $\beta\vDash m$, we view $\bfP_{\alpha\cdot\beta}$, $\bfP_{(\alpha,\beta)}$, and $\bfP_{\alpha\odot\beta}$ as the projective $0$-Hecke modules whose bases are standard tableaux of the ribbon shapes $\diag(\alpha\cdot\beta)$, $\diag(\alpha,\beta)$, and $\diag(\alpha\odot\beta)$, respectively, using the diagram conventions from \Cref{conventions: diagrams}.

\begin{definition}\label{def: del-mu-ses}
First we define the map
$\partial_{\alpha,\beta}: \bfP_{\alpha\cdot\beta} \to \bfP_{(\alpha,\beta)}$.
Let $U$ be a standard ribbon tableau of shape $\diag(\alpha\cdot\beta)$.
Define $\partial_{\alpha,\beta}(U)$ to be the tableau of shape $\diag(\alpha,\beta)$ obtained by translating
the $\beta$-subribbon one unit to the right and keeping the same entries in the corresponding cells.
This produces a standard tableau of $\diag(\alpha,\beta)$ because the translation does not change any
row/column relations within either component.
Since $\partial_{\alpha,\beta}$ is injective on the tableau basis, it is an injective linear map.

Now we define the map $\mu_{\alpha,\beta}: \bfP_{(\alpha,\beta)} \to \bfP_{\alpha\odot\beta}$.
Let $T$ be a standard tableau of shape $\diag(\alpha,\beta)$.
Let $T^{\downarrow}$ denote the filling of $\diag(\alpha\odot\beta)$ obtained by translating the entire $\beta$-component of $T$ one unit down and keeping all entries in the corresponding translated cells.
Define
\begin{equation}\label{eq: mu-map-ses-rewritten}
\mu_{\alpha,\beta}(T)\coloneqq
\begin{cases}
T^{\downarrow} & \text{if $T^{\downarrow}$ is a standard ribbon tableau of shape $\diag(\alpha\odot\beta)$,}\\
0 & \text{otherwise.}
\end{cases}
\end{equation}
Surjectivity of $\mu_{\alpha,\beta}$ is immediate: given any standard tableau $S$ of shape $\diag(\alpha\odot\beta)$, translate the $\beta$-portion of $S$ one unit \emph{up} to obtain a standard tableau $\widetilde S$ of shape $\diag(\alpha,\beta)$, and then $\mu_{\alpha,\beta}(\widetilde S)=S$.
\end{definition}

\begin{lemma}\label{lem: del-mu-equivariant} The maps $\partial_{\alpha,\beta}$ and $\mu_{\alpha,\beta}$ are both $\cH_{n+m}(0)$-equivariant maps.
\end{lemma}

\begin{proof}
It suffices to check that both maps commute with each generator $\ol\pi_i$ under the action of
\Cref{def: hecke-action}.

For $\partial_{\alpha,\beta}$, we only translate the $\beta$-subribbon one unit to the right.
This does not change the row of any entry, so the relative row positions of $i$ and $i+1$ are
preserved, and swapping $i$ and $i+1$ commutes with the translation. Hence
\[
\partial_{\alpha,\beta}(\ol\pi_i U)=\ol\pi_i(\partial_{\alpha,\beta}U)
\qquad\text{for all }U\in\bfP_{\alpha\cdot\beta}.
\]

For $\mu_{\alpha,\beta}$, fix $T\in\bfP_{(\alpha,\beta)}$ and set $S\coloneqq \mu_{\alpha,\beta}(T)$.
If $i$ and $i+1$ lie in the same component of $\diag(\alpha,\beta)$, then translating the $\beta$-component
down does not change their relative row positions, and swapping commutes with translation, so
$\mu_{\alpha,\beta}(\ol\pi_iT)=\ol\pi_iS$.

Now assume $i$ and $i+1$ lie in different components. If after translating the $\beta$-component down the
entries $i$ and $i+1$ are still in different rows, then again their relative row order is unchanged, so
$\ol\pi_i$ commutes with the translation and the same equality holds.

It remains to treat the only exceptional situation: after translation, $i$ and $i+1$ lie in the same row
of $\diag(\alpha\odot\beta)$. In this case $\ol\pi_i(S)=0$ (either $S=0$, or $S$ is standard and $i,i+1$
lie in the same row).
On the source side, if $i$ lies in the $\beta$-component then $\ol\pi_iT=-T$, while if $i$ lies in the
$\alpha$-component then $\ol\pi_iT=s_i(T)$. In either subcase, translating the $\beta$-component down places
the $\beta$-component strictly to the right of the $\alpha$-component in that row, so the translated filling
would have $i$ to the right of $i+1$ and is therefore nonstandard. Hence
$\mu_{\alpha,\beta}(T)=0$ and also $\mu_{\alpha,\beta}(\ol\pi_iT)=0$, proving
$\mu_{\alpha,\beta}(\ol\pi_iT)=\ol\pi_i(\mu_{\alpha,\beta}T)$ in all cases.
\end{proof}

\begin{prop}\label{prop: concatenation-SES}
For $\alpha\in\Comp(n)$ and $\beta\in\Comp(m)$, one has maps
$\partial\coloneqq \partial_{\alpha,\beta}$ and $\mu\coloneqq \mu_{\alpha,\beta}$
giving rise to the following $\cH_{n+m}(0)$-equivariant split short exact sequence
of $\cH_{n+m}(0)$-modules over any field:
\begin{equation}\label{eq: SES-concat}
    0\to \bfP_{\alpha\cdot\beta}\xlongrightarrow{\partial} \bfP_{(\alpha,\beta)}
    \xlongrightarrow{\mu}\bfP_{\alpha\odot\beta} \to 0.
\end{equation}
\end{prop}

\begin{proof}
Equivariance of $\partial_{\alpha,\beta}$ and $\mu_{\alpha,\beta}$ was checked in
\Cref{lem: del-mu-equivariant}. By construction, $\partial_{\alpha,\beta}$ is injective and
$\mu_{\alpha,\beta}$ is surjective.

Let $c_\alpha$ be the northeasternmost cell of the $\alpha$-component and $c_\beta$ the southwesternmost
cell of the $\beta$-component in $\diag(\alpha,\beta)$. Translating the $\beta$-component down introduces
a single new adjacency, namely the horizontal adjacency with
$c_\alpha$ immediately to the left of $c_\beta$. Therefore, for a standard tableau $T$ of shape
$\diag(\alpha,\beta)$, the translated filling $T^\downarrow$ is standard if and only if
$T(c_\alpha)<T(c_\beta)$.

If $U$ is a standard tableau of $\diag(\alpha\cdot\beta)$, then the corresponding cells are vertically
adjacent with $c_\beta$ above $c_\alpha$, so $U(c_\beta)<U(c_\alpha)$. Since $\partial_{\alpha,\beta}$
does not change entries, $\partial_{\alpha,\beta}(U)$ satisfies $T(c_\beta)<T(c_\alpha)$, hence lies in
$\ker(\mu_{\alpha,\beta})$. Thus $\im(\partial_{\alpha,\beta})\subseteq\ker(\mu_{\alpha,\beta})$.

Conversely, if $T\in\bfP_{(\alpha,\beta)}$ satisfies $\mu_{\alpha,\beta}(T)=0$, then
$T(c_\beta)<T(c_\alpha)$. Translating the $\beta$-component one unit left produces a filling of the
connected ribbon $\diag(\alpha\cdot\beta)$, and the only new comparison is the vertical adjacency
$c_\beta$ above $c_\alpha$, which is standard by the same inequality. Hence this filling is a standard
tableau $U$ with $\partial_{\alpha,\beta}(U)=T$, so $\ker(\mu_{\alpha,\beta})\subseteq\im(\partial_{\alpha,\beta})$.

Therefore $\ker(\mu_{\alpha,\beta})=\im(\partial_{\alpha,\beta})$, so \eqref{eq: SES-concat} is exact.
Finally, $\bfP_{\alpha\odot\beta}$ is projective, so $\mu_{\alpha,\beta}$ splits and the sequence is split.
\end{proof}

\begin{example}
Let $\alpha = (2,2)$ and $\beta = (1,2,1)$. Then $\alpha\cdot\beta = (2,2,1,2,1)$ and $\alpha\odot\beta = (2,3,2,1)$.
Here are some examples of the maps $\partial_{\alpha,\beta}$ and $\mu_{\alpha,\beta}$:
\begin{gather}\ytableausetup{boxsize=1em}
\ytableaushort{\none\none\none 4, \none\none 17, \none\none 3, \none 25, 68}
\xrightarrow{\partial_{\alpha,\beta}}
\ytableaushort{\none\none\none\none 4, \none\none\none 17, \none\none\none 3,\none 25, 68}
\xrightarrow{\mu_{\alpha,\beta}} 0 \\
\ytableaushort{\none\none\none\none 4, \none\none\none 17, \none\none\none 5,\none 23, 68}
\xrightarrow{\mu_{\alpha,\beta}}
\ytableaushort{\none\none\none\none 4, \none\none\none 17,\none 235, 68}.
\end{gather}
\end{example}

Before we generalize the complex above, we introduce some notation.

\begin{construction}\label{const: longer-concatenation-complex}
Let $\vec\alpha = (\alpha^{(1)}, \dots, \alpha^{(\ell)})$ be a sequence of compositions with $\abs{\vec\alpha} = n$.
Given an indexing set $I\subseteq [\ell-1]$, denote by $\vec\alpha(I)$ the length $\ell-\abs{I}$ sequence of compositions obtained by replacing the $i$th comma of $\vec \alpha$ with $\odot$ for each $i\in I$.

Define the cochain complex $(\cC(\vec \alpha), \partial)$ of $\cH_n(0)$-modules by
\[
    \cC_i(\vec\alpha) \coloneqq \bigoplus_{\substack{I\subseteq [\ell-1] \\ \abs{I} = i}} \bfP_{\vec\alpha(I)}
\]
and define $\partial_i\colon \cC_i(\vec \alpha) \to \cC_{i+1}(\vec\alpha)$ as follows.
If $I=\{i_1<\cdots<i_r\}$ and $j\in[\ell-1]\setminus I$, set
\[
\sgn(j,I)\coloneqq (-1)^{\#\{t : i_t < j\}}.
\]
For any $T\in \bfP_{\vec\alpha(I)}$, define
\[
\partial_i(T) \coloneqq \sum_{j\in[\ell-1]\setminus I} \sgn(j,I)\,\phi_j(T),
\]
where $\phi_j(T)$ is obtained by applying the gluing operation of \Cref{prop: concatenation-SES}
to the adjacent pair of components $(j,j+1)$ while keeping all other components fixed.
\end{construction}

\begin{theorem}\label{thm: long-ribbon-complex-exact}
For any sequence of compositions $\vec\alpha = (\alpha^{(1)},\alpha^{(2)},\dots, \alpha^{(\ell)})$,
the pair $(\cC(\vec\alpha),\partial)$ is a cochain complex which is acyclic in strictly positive cohomological degrees, and
\[
H^0(\cC(\vec\alpha))\cong \bfP_{\alpha^{(1)}\cdot\alpha^{(2)}\cdot\dots\cdot \alpha^{(\ell)}}.
\]
\end{theorem}
One may think of \Cref{thm: long-ribbon-complex-exact} as ``categorifying'' an identity in $\NSym$
which comes from recursively applying \Cref{eq: product-of-two-ribbons}.

\begin{example}
In the case where $\ell=4$, the exact complex from \Cref{thm: long-ribbon-complex-exact} looks like:
\begin{equation}
\label{lift-of-HG-for-l=4}
	0 \to
	\bfP_{\alpha\cdot\beta\cdot \gamma\cdot\delta}
	\to
	\bfP_{(\alpha,\beta,\gamma,\delta)}
	\xrightarrow[]{\partial_1}
	\begin{matrix}
	\bfP_{(\alpha \odot \beta,\gamma,\delta)} \\
	\oplus\\
	\bfP_{(\alpha,\beta \odot \gamma,\delta)} \\
    \oplus\\
    \bfP_{(\alpha,\beta,\gamma\odot\delta)}
	\end{matrix}
	\xrightarrow{\partial_2}
    	\begin{matrix}
	\bfP_{(\alpha \odot \beta\odot\gamma,\delta)} \\
	\oplus\\
	\bfP_{(\alpha \odot \beta, \gamma\odot\delta)} \\
    \oplus\\
\bfP_{(\alpha,\beta\odot\gamma\odot\delta)}
	\end{matrix}
    \xrightarrow{\partial_3}
	\bfP_{(\alpha \odot \beta \odot \gamma\odot\delta)}
	\to 0.
\end{equation}
The corresponding identity in $\NSym$ generalizing \Cref{eq: product-of-two-ribbons} looks like:
\begin{align*}
R_\alpha R_\beta R_\gamma R_\delta =\;& R_{\alpha\cdot\beta\cdot\gamma\cdot\delta}
+ \bigl(R_{\alpha\odot\beta}R_\gamma R_\delta + R_\alpha R_{\beta\odot \gamma} R_\delta + R_\alpha R_\beta R_{\gamma\odot\delta}\bigr)\\
&-\bigl(R_{\alpha\odot\beta\odot\gamma}R_\delta + R_{\alpha\odot\beta} R_{\gamma\odot\delta} + R_\alpha R_{\beta\odot\gamma\odot\delta}\bigr)
+ R_{\alpha\odot\beta\odot\gamma\odot\delta}.
\end{align*}
In the case where each of the $\alpha^{(i)}$ has length one, note that taking characters of this complex recovers \Cref{eq: R-to-H}.

To see an example of the maps, suppose $\alpha = (2)$, $\beta = (1,2)$, $\gamma = (2,1)$, and $\delta = (1)$. Then we have:
\begin{align*}
\ytableausetup{boxsize = 0.8em}
\ytableaushort{\none\none\none\none\none\none 6, \none\none\none\none\none 8, \none\none\none\none 79, \none\none 34,\none\none 5, 12} \xrightarrow{\partial_1}
&\ytableaushort{\none\none\none\none\none\none 6, \none\none\none\none\none 8, \none\none\none\none 79, \none\none 34,125} \\
&+ \ytableaushort{\none\none\none\none\none\none 6, \none\none\none\none\none 8, \none\none 3479,\none\none 5, 12} \\
\xrightarrow{\partial_2} \;&
\ytableaushort{\none\none\none\none\none\none 6, \none\none\none\none\none 8, \none\none 3479,125}
- \ytableaushort{\none\none\none\none\none\none 6, \none\none\none\none\none 8, \none\none 3479,125}  = 0.
\end{align*}
\end{example}

\begin{proof}
This proof is similar to the proof of \cite{almousa2024veronese}*{Thm.~3.14}, but is in fact simpler.

First we check that $\partial^2=0$.
Fix $I\subseteq[\ell-1]$ and distinct $j<j'$ in $[\ell-1]\setminus I$.
The maps $\phi_j$ are defined by commuting translations of northeast subdiagrams (with the convention
that the output is $0$ if a translation produces a nonstandard filling). Hence for all
$T\in \bfP_{\vec\alpha(I)}$,
\[
\phi_{j'}(\phi_j(T))=\phi_j(\phi_{j'}(T)),
\]
where both sides are understood to be $0$ if either intermediate step is nonstandard.
In $\partial_{i+1}\partial_i(T)$ the pair $(j,j')$ therefore appears twice, once as $j$ then $j'$ and once
as $j'$ then $j$, with opposite signs because
\[
\sgn(j,I)\,\sgn(j',I\cup\{j\}) \;=\; -\,\sgn(j',I)\,\sgn(j,I\cup\{j'\}).
\]
Thus the contributions cancel and $\partial_{i+1}\circ\partial_i=0$ for all $i$, so $\cC(\vec\alpha)$ is a
cochain complex.

We now prove exactness by induction on $\ell$ using a mapping cone.
For $\ell=2$ this is \Cref{prop: concatenation-SES}.
Assume $\ell\ge 3$ and write
\[
\vec\alpha=(\alpha^{(1)},\dots,\alpha^{(\ell)}),\qquad
\vec\alpha'=(\alpha^{(1)},\dots,\alpha^{(\ell-1)}),\qquad
\vec\alpha''=(\alpha^{(1)},\dots,\alpha^{(\ell-1)}\odot\alpha^{(\ell)}).
\]
Decompose the summands of $\cC_i(\vec\alpha)$ according to whether $\ell-1\in I$:
if $\ell-1\notin I$ we obtain the summand $\bfP_{(\vec\alpha'(I),\alpha^{(\ell)})}$, while if
$I=I'\cup\{\ell-1\}$ we obtain $\bfP_{\vec\alpha''(I')}$.
Equivalently, for each $i$ there is an isomorphism
\[
\cC_i(\vec\alpha)\ \cong\ 
\Bigl(\cC_i(\vec\alpha')\otimes \bfP_{\alpha^{(\ell)}}\Bigr){\uparrow}
\ \oplus\
\cC_{i-1}(\vec\alpha'').
\]

Under this identification, the differential on $\cC(\vec\alpha)$ has the standard mapping-cone form:
the terms with $j\le \ell-2$ restrict to the differentials on
$\bigl(\cC(\vec\alpha')\otimes \bfP_{\alpha^{(\ell)}}\bigr){\uparrow}$ and on $\cC(\vec\alpha'')$,
while the terms with $j=\ell-1$ give a morphism of complexes
\[
\Phi:\ \Bigl(\cC(\vec\alpha')\otimes \bfP_{\alpha^{(\ell)}}\Bigr){\uparrow}\longrightarrow \cC(\vec\alpha'')
\]
whose components are the gluing maps $\phi_{\ell-1}$ (with the usual sign $(-1)^i$ in degree $i$).
Thus $\cC(\vec\alpha)\cong \cone(\Phi)$, and we have a short exact sequence of complexes
\begin{equation}\label{eq: cone-ses}
0\to \cC(\vec\alpha'')[-1]\to \cC(\vec\alpha)\to
\Bigl(\cC(\vec\alpha')\otimes \bfP_{\alpha^{(\ell)}}\Bigr){\uparrow}\to 0.
\end{equation}

By the inductive hypothesis, $\cC(\vec\alpha')$ and $\cC(\vec\alpha'')$ are acyclic in positive degrees.
Since $(-)\otimes_\kk \bfP_{\alpha^{(\ell)}}$ and induction are exact, the complex
$\bigl(\cC(\vec\alpha')\otimes \bfP_{\alpha^{(\ell)}}\bigr){\uparrow}$ is also acyclic in positive degrees.
The long exact sequence in cohomology associated to \eqref{eq: cone-ses} then implies that
$\cC(\vec\alpha)$ is acyclic in positive degrees.

Finally, the low-degree part of the long exact sequence identifies $H^0(\cC(\vec\alpha))$ with the kernel
of the connecting map
\[
\Bigl(\bfP_{\alpha^{(1)}\cdots\alpha^{(\ell-1)}}\otimes \bfP_{\alpha^{(\ell)}}\Bigr){\uparrow}
\longrightarrow
\bfP_{\alpha^{(1)}\cdots\alpha^{(\ell-1)}\odot\alpha^{(\ell)}},
\]
which is the near-concatenation map
$\mu_{\alpha^{(1)}\cdots\alpha^{(\ell-1)},\,\alpha^{(\ell)}}$ from \Cref{prop: concatenation-SES}.
Its kernel is $\bfP_{\alpha^{(1)}\cdots\alpha^{(\ell)}}$, completing the induction.
\end{proof}

\section{Koszulness of \texorpdfstring{$\cH_\bullet(0)$}{H*(0)}}\label{sec:koszul}

In this section, we explain how VandeBogert's ribbon Schur module criterion~\cite{vandebogert2025ribbon}
applies to a natural \emph{internally} graded algebra structure built from the tower
$\{\cH_n(0)\}_{n\ge 0}$.
In this section, we must keep track of two gradings simultaneously:
an \emph{external} grading by $n$ (coming from $\cH_n(0)$) and an \emph{internal} grading (the
``column degree'').
Let
\[
\cH_{n_1,\dots,n_r}(0)\coloneqq \cH_{n_1}(0)\otimes_\kk \cdots \otimes_\kk \cH_{n_r}(0)
\subseteq \cH_{n_1+\cdots+n_r}(0)
\]
denote the usual parabolic subalgebra.
For $\cH_n(0)$-modules $M$ and $\cH_m(0)$-modules $N$, define their \defi{induction product}
\begin{equation}\label{eq: induction-product}
M \,\widehat\otimes\, N
\;\coloneqq\;
(M\otimes_\kk N)\uparrow^{\cH_{n+m}(0)}_{\cH_{n,m}(0)}.
\end{equation}
More generally, for $M_i$ a $\cH_{n_i}(0)$-module, write
\[
M_1\widehat\otimes\cdots\widehat\otimes M_r
\;\coloneqq\;
(M_1\otimes_\kk \cdots \otimes_\kk M_r)\uparrow^{\cH_{n_1+\cdots+n_r}(0)}_{\cH_{n_1,\dots,n_r}(0)}.
\]

\begin{construction}\label{construction: weird-grading-koszul-alg}
Fix a field $\kk$.
Define an object
\[
A \;\coloneqq\; \bigoplus_{n\ge 0}\ \bigoplus_{\alpha\vDash n}\bfP_\alpha
\]
where the summand $\bfP_\alpha$ is regarded as a projective $\cH_{|\alpha|}(0)$-module.

We equip $A$ with two gradings:
\begin{itemize}
\item the \emph{external degree} $n$ (the $\cH_n(0)$-module component), and
\item an \emph{internal degree} on indecomposable projectives defined by
\[
\deg_{\mathrm{int}}(\bfP_\alpha)\coloneqq \Len(\alpha^\top),
\]
the number of \textbf{columns} of the ribbon $\diag(\alpha)$.
\end{itemize}
Accordingly, write
\[
A \;=\; \bigoplus_{d\ge 0} A_d,
\qquad
A_d \;=\; \bigoplus_{n\ge 0} A_{d,n},
\qquad
A_{d,n}\coloneqq \bigoplus_{\substack{\alpha\vDash n\\ \Len(\alpha^\top)=d}} \bfP_\alpha.
\]
Thus $A_0=\bfP_{\emptyset}\cong \kk$ (in external degree $0$). In \Cref{lem:A-generated-in-degree-1} we show that $A$ is generated by $A_1$.

We now define a multiplication
\[
m:\ A\widehat\otimes A \longrightarrow A
\]
which is homogeneous for the internal grading and additive for the external grading.
It suffices to define $m$ on summands $\bfP_\alpha\widehat\otimes\bfP_\beta$.
By definition of generalized ribbon projectives,
\[
\bfP_\alpha\widehat\otimes\bfP_\beta
\;\cong\;
(\bfP_\alpha\otimes_\kk \bfP_\beta)\uparrow_{\cH_{|\alpha|,|\beta|}(0)}^{\cH_{|\alpha|+|\beta|}(0)}
\;\cong\;
\bfP_{(\alpha,\beta)}.
\]
Define $m$ on this summand to be the $\cH_{|\alpha|+|\beta|}(0)$-equivariant near-concatenation map
\[
m|_{\bfP_\alpha\widehat\otimes\bfP_\beta}\;:=\;
\mu_{\alpha,\beta}:\ \bfP_{(\alpha,\beta)}\longrightarrow \bfP_{\alpha\odot\beta}\subseteq A.
\]
Extending $\kk$-linearly over all summands defines $m$.

With respect to the internal grading, $m$ is homogeneous because
\[
\Len((\alpha\odot\beta)^\top)=\Len(\alpha^\top)+\Len(\beta^\top),
\]
since near-concatenation glues ribbons horizontally and hence adds the number of columns.
Associativity of $m$ follows from the coherence of the map $\mu_{\alpha,\beta}$ under iterated
near-concatenation, as checked in the proof of \Cref{thm: long-ribbon-complex-exact}. 
\end{construction}

\begin{lemma}\label{lem:A-generated-in-degree-1}
The internally graded algebra $A=\bigoplus_{d\ge 0}A_d$ is generated (under $m$) by $A_1$.
\end{lemma}

\begin{proof}
Let $\alpha$ be a composition and write $\alpha^\top=(c_1,\dots,c_d)$, so $d=\Len(\alpha^\top)$ is the number
of columns of $\diag(\alpha)$.
For each $j$, let $\nu^{(j)}=(1^{c_j})$ be the one-column ribbon of size $c_j$.
Then gluing these columns from left to right gives
\[
\nu^{(1)}\odot \nu^{(2)}\odot \cdots \odot \nu^{(d)}=\alpha.
\]
Since each $\bfP_{\nu^{(j)}}$ lies in $A_1$, iterated multiplication in $A$ produces $\bfP_\alpha$.
\end{proof}

\subsection{Ribbon Schur modules}
Following VandeBogert~\cite{vandebogert2025ribbon}, let $B=\bigoplus_{d\ge 0}B_d$ be a standard internally
graded $\kk$-algebra with multiplication maps
$m:B_a\widehat\otimes B_b\to B_{a+b}$.
For any sequence $\alpha=(\alpha_1,\dots,\alpha_r)$ of positive integers, define the \emph{ribbon Schur module}
$\mathbb{S}_B^\alpha$ to be the kernel of the map
\begin{equation}\label{eq: ribbon-schur-kernel-def}
B_{\alpha_1}\widehat\otimes \cdots \widehat\otimes B_{\alpha_r}
\longrightarrow
\bigoplus_{i=1}^{r-1}
B_{\alpha_1}\widehat\otimes \cdots \widehat\otimes B_{\alpha_i+\alpha_{i+1}} \widehat\otimes \cdots \widehat\otimes B_{\alpha_r},
\end{equation}
whose $i$th component is induced by the multiplication $m:B_{\alpha_i}\widehat\otimes B_{\alpha_{i+1}}\to
B_{\alpha_i+\alpha_{i+1}}$ (tensored with identities on the other factors).

We apply this construction to the internally graded algebra $A$ from
\Cref{construction: weird-grading-koszul-alg}.
Because $A$ is also externally graded (by $n$), each ribbon Schur module $\mathbb{S}_A^\alpha$ inherits an
external grading; we write $(\mathbb{S}_A^\alpha)_{(\bullet,N)}$ for its external degree $N$ summand, which is
a $\cH_N(0)$-module.

\begin{prop}\label{prop:Schur-for-A-decomp}
Let $A$ be the internally graded algebra from \Cref{construction: weird-grading-koszul-alg}, and fix
a composition $\alpha=(\alpha_1,\dots,\alpha_r)$.
For each external degree $N$, there is an isomorphism of projective $\cH_N(0)$-modules
\[
(\mathbb{S}_A^\alpha)_{(\bullet,N)}
\;\cong\;
\bigoplus_{\substack{\gamma^{(1)},\dots,\gamma^{(r)}\\ \Len((\gamma^{(j)})^\top)=\alpha_j\\ \sum_j|\gamma^{(j)}|=N}}
\bfP_{\gamma^{(1)}\cdot\gamma^{(2)}\cdots\gamma^{(r)}}.
\]
Moreover, under the identification
$\cC_0(\vec\gamma)=\bfP_{(\gamma^{(1)},\dots,\gamma^{(r)})}$ from
\Cref{const: longer-concatenation-complex},
the summand indexed by $\vec\gamma$ is the submodule
$H^0(\cC(\vec\gamma))\subseteq \bfP_{(\gamma^{(1)},\dots,\gamma^{(r)})}$.
\end{prop}

\begin{proof}
Fix $\alpha=(\alpha_1,\dots,\alpha_r)$ and an external degree $N$.
From \Cref{construction: weird-grading-koszul-alg},
\[
A_{\alpha_j}
=\bigoplus_{n\ge 0}\ \bigoplus_{\substack{\gamma\vDash n\\ \Len(\gamma^\top)=\alpha_j}} \bfP_\gamma
\qquad (1\le j\le r).
\]
Taking the induction product $A_{\alpha_1}\widehat\otimes\cdots\widehat\otimes A_{\alpha_r}$ and then
extracting external degree $N$ amounts to choosing compositions
$\vec\gamma=(\gamma^{(1)},\dots,\gamma^{(r)})$ with $\Len((\gamma^{(j)})^\top)=\alpha_j$ and
$\sum_j|\gamma^{(j)}|=N$, and then inducing along the tower. Concretely, each such choice yields the
generalized ribbon projective
\[
\bfP_{(\gamma^{(1)},\dots,\gamma^{(r)})}
\cong
\bigl(\bfP_{\gamma^{(1)}}\otimes_\kk\cdots\otimes_\kk \bfP_{\gamma^{(r)}}\bigr)
\uparrow^{\cH_N(0)}_{\cH_{|\gamma^{(1)}|,\dots,|\gamma^{(r)}|}(0)}.
\]
Hence
\begin{equation}\label{eq:domain-decomp}
\bigl(A_{\alpha_1}\widehat\otimes\cdots\widehat\otimes A_{\alpha_r}\bigr)_{(\bullet,N)}
\cong
\bigoplus_{\substack{\gamma^{(1)},\dots,\gamma^{(r)}\\ \Len((\gamma^{(j)})^\top)=\alpha_j\\ \sum_j|\gamma^{(j)}|=N}}
\bfP_{(\gamma^{(1)},\dots,\gamma^{(r)})}.
\end{equation}

By definition \eqref{eq: ribbon-schur-kernel-def} (with $B=A$), $\mathbb{S}_A^\alpha$ is the kernel of the
map whose $i$th component is induced by multiplication
$A_{\alpha_i}\widehat\otimes A_{\alpha_{i+1}}\to A_{\alpha_i+\alpha_{i+1}}$.
For our algebra $A$, on the summand
$\bfP_{\gamma^{(i)}}\widehat\otimes\bfP_{\gamma^{(i+1)}}\cong \bfP_{(\gamma^{(i)},\gamma^{(i+1)})}$
this multiplication is exactly the near-concatenation map
\[
\mu_{\gamma^{(i)},\gamma^{(i+1)}}:\ \bfP_{(\gamma^{(i)},\gamma^{(i+1)})}
\to \bfP_{\gamma^{(i)}\odot\gamma^{(i+1)}}
\qquad\text{(\Cref{prop: concatenation-SES}).}
\]

Since multiplication in $A$ is defined summandwise and extended by direct sum, the defining map for
$\mathbb{S}_A^\alpha$ preserves the index $\vec\gamma$: under the decomposition
\eqref{eq:domain-decomp} of the source (and the analogous decomposition of the target in external
degree $N$), the image of the summand $\bfP_{(\gamma^{(1)},\dots,\gamma^{(r)})}$ lies in
$\bigoplus_{i=1}^{r-1}\bfP_{\vec\gamma(\{i\})}$.
On this summand, the defining map agrees with the degree-$0$ differential
\[
d_0^{\vec\gamma}:\ \bfP_{(\gamma^{(1)},\dots,\gamma^{(r)})}\longrightarrow
\bigoplus_{i=1}^{r-1}\bfP_{\vec\gamma(\{i\})}
\]
in the complex $\cC(\vec\gamma)$ from \Cref{const: longer-concatenation-complex}.
Hence the kernel splits over $\vec\gamma$:
\[
(\mathbb{S}_A^\alpha)_{(\bullet,N)}
=\ker\!\Bigl(\bigoplus_{\vec\gamma} d_0^{\vec\gamma}\Bigr)
\cong
\bigoplus_{\vec\gamma}\ker(d_0^{\vec\gamma})
=
\bigoplus_{\vec\gamma} H^0(\cC(\vec\gamma)).
\]
By \Cref{thm: long-ribbon-complex-exact}, $H^0(\cC(\vec\gamma))\cong
\bfP_{\gamma^{(1)}\cdot\gamma^{(2)}\cdots\gamma^{(r)}}$, which gives the stated decomposition.
\end{proof}

\begin{theorem}\label{thm: koszul}
The internally graded algebra $A$ from \Cref{construction: weird-grading-koszul-alg} is Koszul.
\end{theorem}

\begin{proof}
By \cite{vandebogert2025ribbon}*{Thm~1.3}, it suffices to show that for all compositions $\alpha,\beta$
the canonical sequence
\[
0 \to \mathbb{S}_{A}^{\alpha\cdot\beta} \to \mathbb{S}_{A}^{\alpha}\widehat\otimes \mathbb{S}_{A}^{\beta}
\to \mathbb{S}_{A}^{\alpha\odot\beta} \to 0
\]
is exact.

Since $A$ is externally graded and the multiplication $m$ preserves external degree, each ribbon
Schur module $\mathbb{S}_A^\gamma$ is externally graded, and the maps in the above sequence are
homogeneous for the external grading. Therefore, it suffices to check exactness in each external degree $N$.

Fix an external degree $N$.
By \Cref{prop:Schur-for-A-decomp}, each term in external degree $N$ decomposes as a direct sum of
indecomposable projective $\cH_N(0)$-modules indexed by generalized ribbons.
Moreover, VandeBogert’s canonical sequence is defined functorially from the multiplication of $A$,
and on each summand
\[
\bfP_{(\eta,\theta)} \subseteq
(\mathbb{S}_A^{\alpha}\widehat\otimes \mathbb{S}_A^{\beta})_{(\bullet,N)}
\]
the induced maps agree with the near-concatenation maps
$\partial_{\eta,\theta}$ and $\mu_{\eta,\theta}$ from
\Cref{prop: concatenation-SES}.
Consequently, the external degree $N$ strand of VandeBogert’s sequence is a direct sum of split
short exact sequences
\[
0\to \bfP_{\eta\cdot\theta}\xrightarrow{\ \partial_{\eta,\theta}\ }
\bfP_{(\eta,\theta)}\xrightarrow{\ \mu_{\eta,\theta}\ }\bfP_{\eta\odot\theta}\to 0,
\]
and is therefore exact.

Since $N$
was arbitrary, the canonical sequence is exact, and hence $A$ is Koszul by
\cite{vandebogert2025ribbon}*{Thm~1.3}.
\end{proof}

\section{Exact complexes for skew projective modules}\label{sec:skew-projectives}

In this section we define skew projective modules via the projective--simple pairing, identify their noncommutative characteristics with skew operators on $\NSym$, and deduce concrete descriptions in the cases of skewing by a single row or a single column. We then specialize to skewing by a single box and obtain short exact sequences modeling branching.

\subsection{Skew projectives and their characteristics}
The modules $\bfP_{\alpha/\beta}$ are designed so that taking noncommutative Frobenius characteristics turns them into the standard skewing operators $L_\beta^\perp$ on $\NSym$.

\begin{definition}\label{def: skew-projective}
Fix $\alpha\vDash n$ and $\beta\vDash k\le n$.
Let $\bfC_\beta$ denote the $1$-dimensional simple $\cH_k(0)$-module indexed by $\beta$, and let
\[
\bfC_\beta^\ast \coloneqq \hom_\kk(\bfC_\beta,\kk)
\]
be its $\kk$-linear dual. View $\bfC_\beta^\ast$ as a \emph{right} $\cH_k(0)$-module via
\[
(\varphi\cdot a)(v)\coloneqq \varphi(a\cdot v)
\qquad
(\varphi\in \bfC_\beta^\ast,\ a\in \cH_k(0),\ v\in \bfC_\beta).
\]
Restrict the $\cH_n(0)$-module $\bfP_\alpha$ to the parabolic subalgebra
$\cH_{k,n-k}(0)\cong \cH_k(0)\otimes_\kk \cH_{n-k}(0)$, and denote this restricted module by $\bfP_\alpha\downarrow^{\cH_n(0)}_{\cH_{k,n-k}(0)}.$
Define the \defi{skew projective module} to be the left $\cH_{n-k}(0)$-module
\[
\bfP_{\alpha/\beta}
\ \coloneqq\
\bfC_\beta^\ast\otimes_{\cH_k(0)}
\Bigl(\bfP_\alpha\downarrow^{\cH_n(0)}_{\cH_{k,n-k}(0)}\Bigr).
\]
Here $\cH_{n-k}(0)$ acts through the $\cH_{n-k}(0)$-factor of $\cH_{k,n-k}(0)$ on
$\bfP_\alpha\downarrow_{\cH_{k,n-k}(0)}^{\cH_n(0)}$, and this action commutes with the right
$\cH_k(0)$-action used in the tensor product.
\end{definition}

\begin{prop}\label{prop: skew-projective-represents}
Let $\bfP_{\alpha/\beta}$ be as in \Cref{def: skew-projective}. Then for every finite-dimensional
left $\cH_{n-k}(0)$-module $M$ there is a natural isomorphism
\[
\hom_{\cH_{n-k}(0)}(\bfP_{\alpha/\beta},M)
\ \cong\
\hom_{\cH_n(0)}\!\Bigl(\bfP_\alpha,\ (\bfC_\beta\otimes M)
\uparrow^{\cH_n(0)}_{\cH_{k,n-k}(0)}\Bigr).
\]
In particular, $\bfP_{\alpha/\beta}$ is projective.
\end{prop}

\begin{proof}
Fix a finite-dimensional left $\cH_{n-k}(0)$-module $M$.
View $\bfC_\beta\otimes M$ as a left $\cH_{k,n-k}(0)$-module via the $\cH_k(0)$-action on $\bfC_\beta$
and the $\cH_{n-k}(0)$-action on $M$. Frobenius reciprocity then gives a natural isomorphism
\begin{equation}\label{eq:FR-skew-final}
\hom_{\cH_n(0)}\!\Bigl(\bfP_\alpha,\ (\bfC_\beta\otimes M)\uparrow^{\cH_n(0)}_{\cH_{k,n-k}(0)}\Bigr)
\ \cong\
\hom_{\cH_{k,n-k}(0)}\!\Bigl(\bfP_\alpha\downarrow^{\cH_n(0)}_{\cH_{k,n-k}(0)},\ \bfC_\beta\otimes M\Bigr).
\end{equation}

Now identify $\cH_{k,n-k}(0)\cong \cH_k(0)\otimes_\kk \cH_{n-k}(0)$.
Let $\bfC_\beta^\ast=\hom_\kk(\bfC_\beta,\kk)$, viewed as a right $\cH_k(0)$-module via
\[
(\varphi\cdot a)(v)\coloneqq \varphi(a\cdot v)
\qquad (\varphi\in \bfC_\beta^\ast,\ a\in \cH_k(0),\ v\in \bfC_\beta).
\]
Since $\bfC_\beta$ is finite-dimensional, there is a natural $\cH_{k,n-k}(0)$-module isomorphism
\[
\bfC_\beta\otimes M \;\cong\; \hom_\kk(\bfC_\beta^\ast,M),
\]
where $\cH_k(0)$ acts on $\hom_\kk(\bfC_\beta^\ast,M)$ by precomposition on $\bfC_\beta^\ast$ and
$\cH_{n-k}(0)$ acts by its action on $M$. Therefore
\begin{equation}\label{eq:skew-HTA-step1}
\hom_{\cH_{k,n-k}(0)}\!\Bigl(\bfP_\alpha\downarrow^{\cH_n(0)}_{\cH_{k,n-k}(0)},\ \bfC_\beta\otimes M\Bigr)
\ \cong\
\hom_{\cH_{k,n-k}(0)}\!\Bigl(\bfP_\alpha\downarrow^{\cH_n(0)}_{\cH_{k,n-k}(0)},\ \hom_\kk(\bfC_\beta^\ast,M)\Bigr).
\end{equation}
Applying the usual Hom--tensor adjunction over $\cH_k(0)$ (and keeping $\cH_{n-k}(0)$-equivariance)
yields a natural isomorphism
\begin{equation}\label{eq:skew-HTA-step2}
\hom_{\cH_{k,n-k}(0)}\!\Bigl(\bfP_\alpha\downarrow^{\cH_n(0)}_{\cH_{k,n-k}(0)},\ \hom_\kk(\bfC_\beta^\ast,M)\Bigr)
\ \cong\
\hom_{\cH_{n-k}(0)}\!\Bigl(\bfC_\beta^\ast\otimes_{\cH_k(0)}\bigl(\bfP_\alpha\downarrow^{\cH_n(0)}_{\cH_{k,n-k}(0)}\bigr),\ M\Bigr).
\end{equation}
By \Cref{def: skew-projective}, the $\cH_{n-k}(0)$-module
$\bfC_\beta^\ast\otimes_{\cH_k(0)}\bigl(\bfP_\alpha\downarrow^{\cH_n(0)}_{\cH_{k,n-k}(0)}\bigr)$ is
precisely $\bfP_{\alpha/\beta}$. Combining \eqref{eq:FR-skew-final}, \eqref{eq:skew-HTA-step1}, and
\eqref{eq:skew-HTA-step2} gives the desired natural isomorphism
\[
\hom_{\cH_{n-k}(0)}(\bfP_{\alpha/\beta},M)
\ \cong\
\hom_{\cH_n(0)}\!\Bigl(\bfP_\alpha,\ (\bfC_\beta\otimes M)\uparrow^{\cH_n(0)}_{\cH_{k,n-k}(0)}\Bigr).
\]

Finally, the right-hand side is exact as a functor of $M$: induction
$(-)\uparrow^{\cH_n(0)}_{\cH_{k,n-k}(0)}$ is exact because $\cH_n(0)$ is free as a right
$\cH_{k,n-k}(0)$-module, and $\hom_{\cH_n(0)}(\bfP_\alpha,-)$ is exact because $\bfP_\alpha$ is
projective. Hence $\hom_{\cH_{n-k}(0)}(\bfP_{\alpha/\beta},-)$ is exact, so $\bfP_{\alpha/\beta}$ is
projective.
\end{proof}

\begin{prop}\label{prop: skew-projective-character}
For $\alpha\vDash n$ and $\beta\vDash k$ with $k\le n$, the noncommutative Frobenius characteristic
satisfies
\[
\chN(\bfP_{\alpha/\beta}) \;=\; L_\beta^\perp\, R_\alpha \qquad\text{in }\NSym.
\]
\end{prop}

\begin{proof}
Write
\[
\chN(\bfP_{\alpha/\beta})=\sum_{\gamma\vDash n-k} c_\gamma\,R_\gamma.
\]
By definition of $\chN$, we have that $c_\gamma$ is the
multiplicity of $\bfP_\gamma$ as a direct summand of $\bfP_{\alpha/\beta}$. Since
$\Top(\bfP_\gamma)=\bfC_\gamma$ and $\End(\bfC_\gamma)=\kk$, this multiplicity is (see \cite{Benson1991}*{Lemma 1.7.6})
\begin{equation}\label{eq:coeff-as-hom}
c_\gamma=\dim_\kk \hom_{\cH_{n-k}(0)}(\bfP_{\alpha/\beta},\bfC_\gamma).
\end{equation}

Applying \Cref{prop: skew-projective-represents} with $M=\bfC_\gamma$ gives
\[
\hom_{\cH_{n-k}(0)}(\bfP_{\alpha/\beta},\bfC_\gamma)
\ \cong\
\hom_{\cH_n(0)}\!\Bigl(\bfP_\alpha,\ (\bfC_\beta\otimes \bfC_\gamma)\uparrow\Bigr),
\]
hence, by the projective--simple pairing,
\[
c_\gamma=\bigl\langle \bfP_\alpha,\ (\bfC_\beta\otimes \bfC_\gamma)\uparrow \bigr\rangle.
\]
Sending this identity through the Frobenius characteristic maps and using that induction
corresponds to multiplication in $\QSym$ \cites{KT,bergeron-li} yields
\[
c_\gamma=\langle R_\alpha,\ L_\beta L_\gamma\rangle.
\]
Finally, adjointness of $L_\beta^\perp$ with respect to the $\NSym$--$\QSym$ pairing gives
\[
c_\gamma=\langle L_\beta^\perp R_\alpha,\ L_\gamma\rangle,
\]
so $c_\gamma$ is the coefficient of $R_\gamma$ in $L_\beta^\perp R_\alpha$.
Since this holds for all $\gamma$, we conclude $\chN(\bfP_{\alpha/\beta})=L_\beta^\perp R_\alpha$. 
\end{proof}

\begin{remark}\label{rem: skew-projective-shuffle}
Equivalently, the multiplicity of $\bfP_\gamma$ in $\bfP_{\alpha/\beta}$ is the coefficient of
$L_\alpha$ in the product $L_\beta L_\gamma$ in $\QSym$ (via the pairing between $\NSym$ and $\QSym$).
Using the shuffle rule for products of fundamental quasisymmetric functions, one may interpret this
multiplicity as a shuffle count.
\end{remark}

\subsection{The case \texorpdfstring{$\beta=(1)$}{b = (1)} and branching}
In the case $\beta=(1)$, the operator $L_{(1)}^\perp$ models restriction from $\cH_n(0)$ to the parabolic subalgebra $\cH_{1,n-1}(0)$.

\begin{lemma}\label{lemma: skew-by-one-is-branching}
Fix $\downarrow\coloneqq \downarrow_{\cH_{1,n-1}(0)}^{\cH_n(0)}$.
Then the following diagram commutes:
\[
\begin{tikzcd}[column sep =small] 
K_0(\cH_\bullet(0)) \arrow{rr}{\chN} \arrow{d}{(-){\downarrow}} & &
\NSym \arrow{d}{L_{(1)}^\perp \cdot} \\
K_0(\cH_\bullet(0)) \arrow{rr}{\chN} 
& & \NSym
\end{tikzcd}
\]
\end{lemma}

\begin{proof}
By work of Norton \cite{Norton-Hecke}, the $0$-Hecke algebra $\cH_n(0)$ is free (hence projective) as a right module over any parabolic
subalgebra $\cH_{k,n-k}(0)$.
Consequently, restriction along $\cH_{1,n-1}(0)\subset\cH_n(0)$ sends projective modules to
projective modules, so the left vertical map is well defined on $K_0$.

For $\alpha\vDash n$ and $\gamma\vDash n-1$, Frobenius reciprocity gives
\[
\hom_{\cH_n(0)}\!\left(\bfP_\alpha,\ (\bfC_{(1)}\otimes\bfC_\gamma)\uparrow\right)
\cong
\hom_{\cH_{n-1}(0)}\!\left(\bfP_\alpha\downarrow,\ \bfC_\gamma\right).
\]
Translating this identity via the Frobenius characteristic maps and the pairing
\eqref{eq: nsym-qsym-pairing} yields the stated commutativity.
\end{proof}

\subsection{Row and column skewing}
We now give explicit decompositions of $\bfP_{\alpha/\beta}$ in the cases where $\beta$ is a
single row or a single column.

\begin{notation}\label{not: rowDelete}
Let $\alpha\in\weakComp(n)$ and let $\beta\in\weakComp(k,\Len(\alpha))$ with $\beta_i\le \alpha_i$ for all $i$.
Define $\rowDelete(\alpha,\beta)$ to be the generalized ribbon obtained from $\diag(\alpha)$ by removing
$\beta_i$ boxes from left to right in row $i$ (counting rows from the bottom).
\end{notation}

\begin{prop}\label{prop: skew-by-row-col-mods}
Let $\alpha\in\Comp(n)$ and $0<k<n$.
Then, in $K_0(\cH_{n-k}(0))$ (equivalently, up to isomorphism of projective $\cH_{n-k}(0)$-modules),
\begin{equation}\label{eq: skew-by-row}
    \bfP_{\alpha/(k)}
    \;\cong\;
    \bigoplus_{\substack{\beta\in\weakComp(k,\Len(\alpha))\\ (\alpha-\beta)_i>0\ \text{for}\ i<\Len(\alpha)}}\bfP_{\rowDelete(\alpha,\beta)}.
\end{equation}
Similarly,
\begin{equation}\label{eq: skew-by-column}
    \bfP_{\alpha/(1^k)}
    \;\cong\;
    \bigoplus_{\substack{\beta\in\weakComp(k,\Len(\alpha^\top))\\ (\alpha^\top-\beta)_i>0\ \text{for}\ i<\Len(\alpha^\top)}}\bfP_{\rowDelete(\alpha^\top,\beta)^\top}.
\end{equation}
\end{prop}

\begin{remark}
In \eqref{eq: skew-by-row}, the condition $(\alpha-\beta)_i>0$ for $i<\Len(\alpha)$ says that we may delete an entire row of $\alpha$ only if it is the \emph{top} row.
In \eqref{eq: skew-by-column}, the analogous condition says we may delete an entire row of $\alpha^\top$ only if it is the top row, i.e.\ we may delete an entire column of $\alpha$ only if it is the \emph{leftmost} column.
\end{remark}

\begin{proof}
We prove \eqref{eq: skew-by-row}; the column case follows by transposing ribbons and using that $(1^k)$ corresponds to the unique permutation $k\,\cdots\,2\,1$.

By \Cref{prop: skew-projective-character}, $\bfP_{\alpha/(k)}$ is determined in $K_0$ by the coefficients of $L_{(k)}^\perp R_\alpha$.
Equivalently, for each $\gamma\vDash n-k$, the multiplicity of $\bfP_\gamma$ in $\bfP_{\alpha/(k)}$ is the coefficient of $L_\alpha$ in $L_{(k)}L_\gamma$ (see \Cref{rem: skew-projective-shuffle}).

Choose the unique word $w_{(k)}=12\cdots k$ with descent composition $(k)$.
By the shuffle product rule in \Cref{sec:hopf}, the coefficient of $L_\alpha$ in $L_{(k)}L_\gamma$ counts shuffles of $w_{(k)}$ with a word of descent composition $\gamma$ that produce a word of descent composition $\alpha$.
Interpreting these shuffles on ribbon tableaux of shape $\alpha$, specifying an occurrence of the increasing subword $12\cdots k$ is equivalent to choosing, in each row $i$ of $\diag(\alpha)$, a number $\beta_i$ of boxes (from the left) occupied by letters from $\{1,\dots,k\}$.
Interpreting shuffles on standard ribbon tableaux of shape $\alpha$, specifying an occurrence of the increasing word $12\cdots k$ amounts to choosing which $k$ cells of $\diag(\alpha)$ are filled by the letters $\{1,\dots,k\}$.
Because the tableau is row- and column-strict, the set of cells occupied by $\{1,\dots,k\}$ is left-justified in each row.
Moreover, choosing all cells in a non-top row would disconnect the remaining ribbon, so we impose $(\alpha-\beta)_i>0$ for $i<\Len(\alpha)$.

Removing these $\beta_i$ boxes produces the generalized ribbon $\rowDelete(\alpha,\beta)$, filled by the remaining letters $\{k+1,\dots,n\}$.
Varying the shuffle on the remaining letters produces precisely the projective module $\bfP_{\rowDelete(\alpha,\beta)}$.
Summing over all admissible $\beta$ yields the decomposition \eqref{eq: skew-by-row} in $K_0(\cH_{n-k}(0))$.
Since $K_0$ of projectives is free on indecomposable projectives, the equality in $K_0$ upgrades to an isomorphism of projective modules.
\end{proof}

We can combine \Cref{prop: skew-by-row-col-mods} and the complex in \Cref{thm: long-ribbon-complex-exact} to better understand the structure of $\bfP_{\alpha/\beta}$ in the case where $\beta = (k)$ or $(1^k)$. The following examples suggest how to do so.

\begin{example}
    Set $\alpha = (1,3,2)$ and $\beta = (3)$. According to \Cref{prop: skew-by-row-col-mods}, in order to understand the structure of $\bfP_{\alpha / \beta}$, we should
    consider the following diagrams coming from removing the gray boxes (corresponding to $\beta$) from $\alpha$:
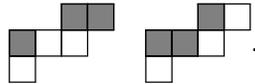
    \[
        \ydiagram{2+2,3,1}*[*(gray)]{2+2,0+1,1+0} \quad 
        \ydiagram{2+2,3,1}*[*(gray)]{2+1,0+2}.
    \]
    Note that the left diagram is allowed because the only time we are allowed to delete an entire row is if it is the topmost row.
    Therefore, we have that
    \[
    \bfP_{\alpha/\beta} \cong \bfP_{((1),(2))} \oplus \bfP_{((1),(1),(1))}.
    \]
    Let $\vec\gamma = ((1),(2))$ and $\vec\delta = ((1),(1),(1))$. We can take the direct sum $\cC(\vec\gamma) \oplus \cC(\vec\delta)$ of the complexes from \Cref{const: longer-concatenation-complex} to obtain the complex
\begin{equation*}
	0 \to 
    \begin{matrix}
	\bfP_{(1,2)} \\
    \oplus \\
    \bfP_{(1,1,1)}
    \end{matrix}
	\to 
    \begin{matrix}
        \bfP_{\alpha / \beta} \\
        \Vert \\
        \bfP_{((1),(2))} \\
        \oplus \\
        \bfP_{((1),(1),(1))}
    \end{matrix}
	\to
	\begin{matrix}
	\bfP_{(3)} \\
	\oplus\\
	\bfP_{((1),(2))} \\
    \oplus\\
    \bfP_{((2),(1))}
	\end{matrix}
	\to
    \bfP_{(3)}
	\to 0,
\end{equation*}
    which is exact because it is the direct sum of two exact complexes.
    Therefore, we have that $\bfP_{\alpha / \beta} \cong \bfP_{(2,1)} \oplus \bfP_{(1,1,1)} \oplus \bfP_{((1),(2))} \oplus \bfP_{((2),(1))}$.
    Applying the concatenation/near-concatenation identity one more time, we find that the complete decomposition of $\bfP_{\alpha/\beta}$ into indecomposable projectives is
    \[
    \bfP_{(1,3,2)/(3)} \cong \bfP_{(1,2)}^2 \oplus \bfP_{(1,1,1)} \oplus \bfP_{(3)}^2 \oplus \bfP_{(2,1)}.
    \]
\end{example}

\begin{example}
Set $\alpha = (1,3,2)$ and $\beta = (1,1)$. Then $\alpha^\top = \rev(\alpha^c) = (1,2,1,2)$.
By \Cref{prop: skew-by-row-col-mods}, we should consider the following diagrams where we removed the gray boxes (corresponding to $\beta$) from $\alpha$ in order to understand $\bfP_{\alpha/\beta}$:
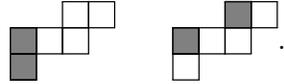
    \[
        \ydiagram{2+2,3,1}*[*(gray)]{4+0,0+1,0+1} \quad 
        \ydiagram{2+2,3,1}*[*(gray)]{2+1,0+1}.
    \]
    Note that the leftmost diagram is allowed because the only time we can remove an entire row from $\alpha^\top$ is if it is the top row, which corresponds to the leftmost column of the original diagram for $\alpha$. Therefore, we have that $\bfP_{\alpha/(1,1)} = \bfP_{(2,2)} \oplus \bfP_{((1),(2),(1))}$. Applying \Cref{thm: long-ribbon-complex-exact} to the diagram $\vec\gamma = ((1),(2),(1))$ iteratively, we obtain the diagram
\[
\begin{tikzcd}
	&&& 0 \\
	&&& {\mathbf{P}_{(1,3)}\oplus\mathbf{P}_{(3,1)}} \\
	0 & {\mathbf{P}_{(1,2,1)}} & {\mathbf{P}_{((1),(2),(1))}} & {\mathbf{P}_{((1),(3))}\oplus \mathbf{P}_{((3),(1))}} & {\mathbf{P}_{(4)}} & 0 \\
	&&& {\mathbf{P}_{(4)}^2} \\
	&&& 0
	\arrow[from=1-4, to=2-4]
	\arrow[from=2-4, to=3-4]
	\arrow[from=3-1, to=3-2]
	\arrow[from=3-2, to=3-3]
	\arrow[from=3-3, to=3-4]
	\arrow[from=3-4, to=3-5]
	\arrow[from=3-4, to=4-4]
	\arrow[from=3-5, to=3-6]
	\arrow[from=4-4, to=5-4]
\end{tikzcd}
\]
Combining everything together, we see that
\[
\bfP_{(1,3,2)/(1,1)} \cong \bfP_{(2,2)} \oplus \bfP_{(1,2,1)} \oplus \bfP_{(1,3)} \oplus \bfP_{(3,1)} \oplus \bfP_{(4)}.
\]
\end{example}

\subsection{Skewing by a single box and a branching short exact sequence}

We now specialize to $\beta=(1)$.
Deleting one box from a ribbon can disconnect it into at most two connected components; we make this explicit.

\begin{definition}
For a generalized ribbon diagram $\rho$, let $\components(\rho)$ denote the ordered list of its
connected components, listed from southwest to northeast.
\end{definition}

\begin{prop}\label{prop:ses-projective-reps}
Let $\alpha\vDash n$ and let $\epsilon_i$ denote the $i$th standard basis vector in $\ZZ^{\Len(\alpha)}$.
For each $i\in[\Len(\alpha)]$ with $\alpha_i>1$, the generalized ribbon $\rowDelete(\alpha,\epsilon_i)$ has
at most two connected components. Write its (ordered) list of components as
\[
\components(\rowDelete(\alpha,\epsilon_i)) \;=\;
\begin{cases}
(\alpha_i',\alpha_i'') & \text{if it is disconnected},\\
(\alpha-\epsilon_i) & \text{if it is connected}.
\end{cases}
\]
Then there is a split short exact sequence of $\cH_{1,n-1}(0)$-modules
\[
0 \to \bigoplus_{\alpha_i>1} \bfP_{(1)} \otimes \bfP_{\alpha - \epsilon_i}
\xrightarrow[]{\partial}
\bfP_{\alpha}{\downarrow}^{\cH_n(0)}_{\cH_{1, n-1}(0)} 
\xrightarrow[]{\mu}
\bigoplus_{\alpha^\top_j>1} \bfP_{(1)} \otimes \bfP_{(\alpha^\top - \epsilon_j)^\top}
\to 0,
\]
where on each disconnected summand $\bfP_{(1)}\otimes \bfP_{(\alpha_i',\alpha_i'')}$ the maps are given by
$1\otimes \partial_{\alpha_i',\alpha_i''}$ and $1\otimes \mu_{\alpha_i',\alpha_i''}$ as in \Cref{def: del-mu-ses},
and on connected summands $\bfP_{(1)}\otimes \bfP_{\alpha-\epsilon_i}$ the map $\partial$ is the natural inclusion
of that direct summand in the decomposition of $\bfP_\alpha\downarrow$.
\end{prop}

\begin{proof}
By \Cref{lemma: skew-by-one-is-branching} and \Cref{prop: skew-by-row-col-mods} (with $k=1$),
\[
\bfP_{\alpha}\downarrow_{\cH_{1,n-1}(0)}^{\cH_n(0)}
\ \cong\
\bfP_{(1)} \otimes \bfP_{\alpha/(1)}
\ \cong\
\bfP_{(1)} \otimes \bigoplus_{\alpha_i>1} \bfP_{\rowDelete(\alpha,\epsilon_i)}.
\]

Now fix $i$ with $\alpha_i>1$.
If $i=1$, then $\rowDelete(\alpha,\epsilon_1)=\diag(\alpha-\epsilon_1)$ is connected.
If $i>1$, then deleting the leftmost cell in row $i$ removes the unique vertical adjacency between rows $i-1$ and $i$, hence $\rowDelete(\alpha,\epsilon_i)$ has exactly two connected components; write
\[
\components(\rowDelete(\alpha,\epsilon_i))=(\alpha_i',\alpha_i'').
\]

By \Cref{prop: concatenation-SES}, each such two-component generalized ribbon projective sits in a split short exact sequence
\[
0 \to \bfP_{\alpha_i'\cdot\alpha_i''} \xrightarrow{\ \partial\ } \bfP_{(\alpha_i',\alpha_i'')} \xrightarrow{\ \mu\ } \bfP_{\alpha_i'\odot\alpha_i''} \to 0.
\]
In our situation, $\alpha_i'\cdot\alpha_i''=\alpha-\epsilon_i$, while the near-concatenation $\alpha_i'\odot\alpha_i''$ corresponds to removing one box from a column of $\alpha$, i.e.\ it is of the form $(\alpha^\top-\epsilon_j)^\top$ for the appropriate column index $j$.
Taking the direct sum over all $i$ and tensoring with $\bfP_{(1)}$ yields the claimed short exact sequence, and it splits because all terms are projective.
\end{proof}

\begin{cor}\label{cor: nsym-skew-1-box}
For any $\alpha \in \Comp(n)$ with $n>1$, we have
\begin{equation}\label{eq: NSym-ribbon-branching}
L_{(1)}^\perp R_\alpha = \sum_{\alpha_i>1} R_{\alpha-\epsilon_i} + \sum_{\alpha^\top_j>1} R_{(\alpha^\top-\epsilon_j)^\top},
\end{equation}
where $\epsilon_i$ denotes the $i$th standard basis vector in $\ZZ^{\Len(\alpha)}$.
\end{cor}

\begin{remark}
    See Miller \cite{Miller21} and references therein for related identities in $\Sym$.
\end{remark}

\begin{example}
By \Cref{prop:ses-projective-reps}, the following is a short exact sequence of projective $\cH_{1,4}(0)$-modules:
\[
0 \to \bfP_{(1)}\otimes \left(\bfP_{(1, 2, 1)} \oplus \bfP_{(2,1,1)}\right) \xrightarrow{\partial} \bfP_{(2,2,1)}\big\downarrow^{\cH_5(0)}_{\cH_{1,4}(0)} \xrightarrow{\mu} \bfP_{(1)}\otimes \left(\bfP_{(3,1)} \oplus \bfP_{(2,2)}\right) \to 0.    
\]
On the level of basis elements, we see that under $\partial$: 
\[
        \ytableausetup{centertableaux, boxsize = 1em}
        \ytableaushort{\none1,23,4} \mapsto \ytableaushort{\none\none1,\none23,45}, \qquad
        \ytableaushort{\none1,24,3} \mapsto \ytableaushort{\none\none1,\none24,35}, 
\qquad    \ytableaushort{\none2,13,4} \mapsto \ytableaushort{\none\none2,\none13,45}, \qquad
        \ytableaushort{\none2,14,3} \mapsto \ytableaushort{\none\none2,\none14,35},
\]
\[
        \ytableaushort{\none3,14,2} \mapsto \ytableaushort{\none\none3,\none14,25}, \qquad 
        \ytableaushort{\none2,\none3,14} \mapsto \ytableaushort{\none\none2,\none35,14}, 
\qquad
        \ytableaushort{\none1,\none3,24} \mapsto \ytableaushort{\none\none1,\none35,24}, \qquad
        \ytableaushort{\none1,\none2,34} \mapsto \ytableaushort{\none\none1,\none25,34}.
\]    
Under $\mu$, the elements that map to nonzero elements are: 
\[
        \ytableaushort{\none\none1,\none34,25} \mapsto \ytableaushort{\none\none1,234}, \qquad
        \ytableaushort{\none\none2,\none34,15} \mapsto \ytableaushort{\none\none2,134}, 
\qquad     \ytableaushort{\none\none3,\none24,15} \mapsto \ytableaushort{\none\none3,124}, \qquad 
        \ytableaushort{\none\none2,\none15,34} \mapsto \ytableaushort{\none12,34},
\]
\[
        \ytableaushort{\none\none3,\none15,24} \mapsto \ytableaushort{\none13,24}, \qquad
        \ytableaushort{\none\none4,\none15,23} \mapsto \ytableaushort{\none14,23},
\qquad
        \ytableaushort{\none\none3,\none25,14} \mapsto \ytableaushort{\none23,14}, \qquad
        \ytableaushort{\none\none4,\none25,13} \mapsto \ytableaushort{\none24,13}.
\]
All other standard tableaux of shape $\diag(2, 2, 1)$ go to zero. 
\end{example}

\section*{Acknowledgements}
The authors thank Marcelo Aguiar,
Erin Delargy, Anastasia Nathanson, Vic Reiner, Ramanuja Charyulu Telekicherla Kanadalam, and Keller VandeBogert for helpful conversations. This project began at the 2023 Twin Cities REU in Combinatorics and Algebra, which was supported by NSF grant DMS-1745638.

\bibliographystyle{amsplain}
\bibliography{paperbib}

@article {HeckeTab,
    AUTHOR = {Huang, Jia},
     TITLE = {A tableau approach to the representation theory of 0-{H}ecke algebras},
   JOURNAL = {Annals of Combinatorics},
    VOLUME = {20},
      YEAR = {2016},
     PAGES = {831--868},
       DOI = {10.4007/s00026-016-0338-5},
       URL = {https://doi.org/10.48550/arXiv.1501.05250},
}

@article{KT,
    author = {Krob, Daniel and Thibon, Jean-Yves},
    title = {Noncommutative symmetric functions {IV}: quantum linear groups and {H}ecke algebras at q = 0},
    journal = {Journal of Algebraic Combinatorics},
    volume = {6}, 
    year = {1997}, 
    pages = {339--376},
}

@article{NSym,
    author = {Gelfrand and Krob and Lascoux and Leclerc and Retakh and Thibon},
    title = {Noncommutative symmetric functions},
    year = {1994}
}

@article{grinbergReiner-hopf-survey,
  title={Hopf algebras in combinatorics},
  author={Grinberg, Darij and Reiner, Victor},
  journal={arXiv preprint arXiv:1409.8356},
  year={2014}
}

@article{Norton-Hecke,
  title={0-{H}ecke algebras},
  author={Norton, PN},
  journal={Journal of the Australian Mathematical Society},
  volume={27},
  number={3},
  pages={337--357},
  year={1979},
  publisher={Cambridge University Press}
}

@article{Carter-Hecke,
    author = {Carter, R. W.},
    title = {Representation theory of the 0-{H}ecke algebra},
    year = {1985}
}

@article{bergeron-li,
  title={Algebraic structures on {G}rothendieck groups of a tower of algebras},
  author={Bergeron, Nantel and Li, Huilan},
  journal={Journal of Algebra},
  volume={321},
  number={8},
  pages={2068--2084},
  year={2009},
  publisher={Elsevier}
}

@article{almousa2024veronese,
  title={Equivariant resolutions over {V}eronese rings},
  author={Almousa, Ayah and Perlman, Michael and Pevzner, Alexandra and Reiner, Victor and VandeBogert, Keller},
  journal={Journal of the London Mathematical Society},
  volume={109},
  number={1},
  pages={e12848},
  year={2024},
  publisher={Wiley Online Library}
}

@book{Benson1991, 
place={Cambridge}, 
series={Cambridge Studies in Advanced Mathematics}, 
title={Representations and Cohomology}, 
publisher={Cambridge University Press}, 
author={Benson, D. J.}, year={1991}, 
collection={Cambridge Studies in Advanced Mathematics}
}

@article{Miller21,
    author = {Miller, Alexander R.},
    title = {On Foulkes characters},
    journal = {Mathematische Annalen},
    year = {2021},
    volume = {381},
    pages = {1589-1614},
    doi = {https://doi.org/10.1007/s00208-021-02197-4}
}

@article{vandebogert2025ribbon,
  title={Ribbon Schur functors},
  author={VandeBogert, Keller},
  journal={Algebra \& Number Theory},
  volume={19},
  number={4},
  pages={771--834},
  year={2025},
  publisher={Mathematical Sciences Publishers}
}

@book{polishchuk2005quadratic,
  title={Quadratic algebras},
  author={Polishchuk, Alexander and Positselski, Leonid},
  volume={37},
  year={2005},
  publisher={American Mathematical Soc.}
}

\end{document}